\documentclass{amsart}
\usepackage[utf8]{inputenc}
\usepackage[T1]{fontenc}
\usepackage{esint}
\usepackage{siunitx}
\usepackage[english]{babel}
\usepackage{color}
\usepackage{xcolor}
\usepackage{amsmath} 
\usepackage{amssymb} 
\usepackage{amsthm}
\usepackage{graphicx}
\usepackage[colorlinks=true, linkcolor=blue]{hyperref}
\usepackage{cancel}

\usepackage{lmodern}
\usepackage{microtype}

\theoremstyle{plain}
\newtheorem{thm}{Theorem}[section]
\newtheorem*{thm*}{Theorem}
\newtheorem{prop}[thm]{Proposition} 
\newtheorem{lem}[thm]{Lemma} 
\newtheorem{cor}[thm]{Corollary}
\newtheorem*{cor*}{Corollary}

\newtheorem{defi}[thm]{Definition}

\newcommand {\R} {\mathbb{R}} \newcommand {\Z} {\mathbb{Z}}
\newcommand {\T} {\mathbb{T}} \newcommand {\N} {\mathbb{N}}

\newcommand {\p} {\partial}
\renewcommand {\d} {\partial}
\newcommand {\dt} {\partial_t} 

\newcommand {\Id} {\textrm{Id}}

\DeclareMathOperator{\supp}{supp}
\DeclareMathOperator{\dv}{div}
\DeclareMathOperator{\sech}{sech}
\DeclareMathOperator{\dist}{dist}

\begin{document}
\title{On variable viscosity and enhanced dissipation}

\begin{abstract}
  In this article we consider the 2D Navier-Stokes equations with variable
  viscosity depending on the vertical position.
  As our main result we establish linear enhanced dissipation near the non-affine
  stationary states replacing Couette flow. Moreover it turns out that the
  shear flow overcompensates for weakening viscosity: decreasing viscosity leads to stronger enhanced dissipation and
  increasing viscosity leads to weaker dissipation than in the constant
  viscosity case.
\end{abstract}
\author{Xian Liao}
\address{Karlsruhe Institute of Technology\\ Englerstraße 2\\ 76131 Karlsruhe\\ Germany}
\email{xian.liao@kit.edu}
\author{Christian Zillinger}
\address{Karlsruhe Institute of Technology\\ Englerstraße 2\\ 76131 Karlsruhe\\ Germany}
\email{christian.zillinger@kit.edu}
\date{\today}
\maketitle


\section{Introduction}

In the present paper we are concerned with the  two-dimensional  incompressible
Navier-Stokes equations in the presence of large (stratified) viscosity variations
\begin{align}\label{NS}
\left\{
\begin{aligned} 
&\d_t v+v\cdot \nabla v-\dv(\mu  S v)+\nabla p=0,\\
&\dv v=0.
\end{aligned}
\right.
\end{align}
Here $t\in [0,\infty)$ and $\begin{pmatrix}x\\ y\end{pmatrix}\in \T\times\R$ denote
the time and space variables.
The vector-valued function $v=v(t,x,y): [0,\infty)\times\R^2\to\R^2$ and the
scalar function $p=p(t,x,y): [0,\infty)\times\R^2\to \R$ denote the unknown
velocity vector field and the unknown pressure of the two-dimensional
flow, respectively.
\noindent

The symmetric part of the velocity gradient  
\begin{align*}
  \frac12 Sv:=\frac12 (\nabla v+(\nabla v)^T)
\end{align*}
denotes the symmetric deformation  tensor.
The  viscosity coefficient
\begin{align*}
  \mu=\mu(x,y)
\end{align*}
is a given non-constant positive scalar function.
More precisely, we consider the case of stratified viscosity $\mu(y)$ depending
on the vertical direction only and study its interplay with 2D shear flows.

Viscous stratification is a typical phenomenon not only in nature (e.g. in the
atmosphere and ocean flows) but also in industrial application (e.g. in the chemical and food industry).
The (in)stabilities in  viscosity-stratified flows have attracted constant
interests of physicists \cite{Craik, GS, Heisenberg, Lin, HB, Yih}.
While additional dissipation at first sight suggest stabilization \footnote{The
  Orr-Sommerfeld eigenvalue Problem has only positive eigenvalues for Couette
  flows, which implies the stability of Couette flows for all Reynolds number,
  but experiments showed instability under small but finite perturbations.}, in
experiments viscosity exhibits dual roles \cite[Chapter 8, pp. 160]{Drazin}: a stabilizing role due to the dissipation of energy and a more subtle destabilizing role.
 Yih \cite{Yih} showed that the instability in a low Reynolds number flow can be
 caused by viscosity stratifications  (see also Craik \cite{Craik} for the study of flows with continuous viscosity stratification).
These results motivated decades of active researches on the instability caused by viscosity interfaces, see \cite{GS} for a review paper on this topic.

In this paper we consider the model \eqref{NS} of the fluids with equal
density/temperature but different viscosities, which can for instance be used to describe the transport of the highly viscous oil and an immiscible low viscous lubricant (see e.g. \cite{JRR, PV} for the relevant instability analysis).
We then study the asymptotic behavior of perturbations to the
shear flow solutions 
\begin{align}\label{shear}
\mu=\mu(y),\quad v=\begin{pmatrix} U(y)\\ 0\end{pmatrix},
\end{align}
which satisfies the hydrostatic balance
\begin{align}\label{balance}
\d_y(\mu\d_y U)=0.
\end{align}
This condition implies that as $\mu$ decreases $\p_y U$ increases
and vice versa.

As an additional assumption, while we allow $\mu$ and hence $\p_y U$ to change
by several orders of magnitude, we require that locally there is not too much oscillation: 
\begin{align}\label{bound:mu}
\Bigl\|\p_y \ln(\mu(y)) \Bigr\|_{W^{1,\infty}}\leq 0.001.
\end{align}
Thus, for instance $\mu$ may decrease exponentially but only with a small
exponent $c$.

As a consequence of the balance relation \eqref{balance} one observes that the variable viscosity coefficient changes the slope of the underlying velocity profile, such that the viscous stratification comes into play, even at high Reynolds numbers 
\footnote{The viscosity variations increase the order of the Orr-Sommerfeld
  equation from two to four, which makes a difference in the dynamics even at high
  Reynolds numbers (contrary to the intuitive expectation of negligible viscous
  effect).}. 
  More precisely, as we discuss following Theorem \ref{thm:main} as
the viscosity \emph{decreases} towards zero, the effective dissipation $\mu (\p_yU)^2$ rate becomes
\emph{larger}.
This also helps to explains wall heating or cooling techniques (corresponding to the liquid flows or gas flows respectively) in industrial application, which produce less viscous flow near the wall, and hence  stabilize the flows \cite{BG}.

In recent years there has been extensive research on the stability study of the shear flows \eqref{shear} for the inviscid fluids with
\begin{align*}
  \mu=0,
\end{align*}
and for the viscous fluids with constant viscosity
\begin{align*}
  \mu=\textrm{const.}>0.
\end{align*}
Since the literature is extensive, we here do not provide a complete overview
but refer the interested reader to the following recent works for further discussion
\cite{jia2019linear,lin2019metastability,Zhang2015inviscid,ionescu2018inviscid,widmayer2018convergence,yang2018linear,bedrossian2016enhanced,elgindi2015sharp,Lin-Zeng,WZZkolmogorov,bedrossian2017stability,bedrossian2013inviscid,deng2020stability,Lin-Zeng}.
We, in particular, recall that for linearized equations around Couette flow
\begin{align*}
  \mu=\textrm{const.}>0, U(y)=y,
\end{align*}
it can be shown by explicit calculations that the interplay of shearing and
dissipation leads to damping with a rate
\begin{align*}
  \exp(-C\sqrt[3]{\mu} t),
\end{align*}
and thus on a time scale $\mu^{-\frac{1}{3}}$ much smaller than the dissipation
time scale $\mu^{-1}$. This phenomenon is hence called \emph{enhanced
  dissipation} (see \cite{bedrossian2016sobolev} for further discussion and the
analysis of the nonlinear problem).

Given a particular value of the viscosity at a given point,
$\mu(y_0)$, in this paper we are interested in the change of the (local, effective) dissipation rates if $\mu$ varies as
$y>y_0$ increases. For instance, how much of an increase or decrease of $\mu$ is
required to change the dissipation rate by a factor of $10$?
As our main results we establish stability of the linearized equations and prove
damping with a local rate $\mu(U')^2$, which is \emph{inversely} proportional to $\sqrt[3]{\mu}$. 
\begin{thm}
  \label{thm:main}
  Let $\mu \in C^2(\R)$ with $\mu>0$ be a given stratified viscosity profile.
  Then a stationary solution of the Navier-Stokes equations \eqref{NS} on $\T\times \R$  in the presence of the variable viscosity $\mu$ is given by a shear flow $v=(U(y),0)^T$ such that
  \begin{align}
    \label{eq:hydrobalance}
      \mu \p_y U = \text{const.}
    \end{align}
  and the linearized equations around this solution in vorticity formulation
  read
  \begin{align}
    \label{eq:omega}
    \begin{split}
    \dt \omega + U \p_x \omega &= U'' v_2 + \dv(\mu \nabla \omega) - \dv(\mu' \nabla v_1) - \mu'' v_2, \\
    v&=\nabla^{\bot}(-\Delta)^{-1}\omega,
  \end{split}
  \end{align}
  where $U'=\p_y U$ and $U''=\p_y^2U$ denote $y$ derivatives.

  Additionally suppose that $\mu$ only varies gradually, in the sense that
  \begin{align}
  \label{eq:gradual}
    \begin{split}
    \|\frac{\mu'}{\mu}\|_{L^\infty} + \|U'\p_y \frac{\mu'}{\mu}\|_{L^\infty} &< 0.001, \\
    \frac{1}{\sup{\mu}} \|\mathcal{F}(\mu')\|_{L^1(\R \setminus(-\sup(\mu)^{-1/3}, \sup(\mu)^{-1/3}))}&< 0.001,
  \end{split}
  \end{align}
  where $\mathcal{F}(\mu)$ denotes the (tempered) Fourier transform and that $\mu$ is bounded above and below and sufficiently small so that
  \begin{align}
    \frac{(\sup(\mu))^2}{\inf(\mu)} < 0.1.
  \end{align}
  For instance, $\mu$ may grow at a (small) exponential rate from a value
  $\nu^2$ to $0.1 \nu$ with $\nu<0.1$. 
  For simplicity of presentation also assume that $U'\geq 1$ (which can be
  assumed without loss of generality after rescaling).
  Then the linearized equations \eqref{eq:omega} are stable and exhibit enhanced
  dissipation. More precisely, there exists a time-dependent family of
  operators $A(t)$ with
  \begin{align*}
   0.1 \|\omega(t)\|_{L^2(U' dy)} \leq \|A(t)\omega(t)\|_{L^2(U' dy)}\leq \|\omega(t)\|_{L^2(U' dy)}.
  \end{align*}
  Furthermore, if $\int \omega_0 dx=0$ then for all times $t>0$ it holds that
  \begin{align*}
    \frac{d}{dt} \|A(t)\omega(t)\|_{L^2(U'dy)}^2 \leq -0.001 (\|(\mu (U')^2))^{1/6}\omega\|^2 + \|\sqrt{U'}v\|^2).
  \end{align*}
  Moreover, under further regularity assumptions these results also extend to stability of the ``profile''
  $W(t,x,U(y)):=\omega(t,x-tU(y),y)$ in higher Sobolev norms $H^N$ (see
  Proposition \ref{prop:HN} for a precise statement).
\end{thm}

Let us comment on these results:

\begin{itemize}
\item We remark that due to the balance relation \eqref{eq:hydrobalance}
  it holds that 
  \begin{align*}
    \mu U' =\text{const.} =: \sigma.
  \end{align*}
  Hence, one observes that the local dissipation rate satisfies
  \begin{align*}
    \sqrt[3]{\mu (U')^2}=  \sqrt[3]{\mu (U')^2 \mu \frac{1}{\mu}} =  \sqrt[3]{\sigma^2} \frac{1}{\sqrt[3]{\mu(y)}} 
  \end{align*}
  and thus is proportional to $\mu(y)^{-1/3}$.
  Thus a \emph{decrease} of $\mu$ by a factor $1000$
  corresponds to an \emph{increase} of the dissipation rate by a factor $10$.
  Conversely, \emph{increasing} the viscosity compared to $\mu(y_0)$ corresponds
  to \emph{weaker} dissipation.
\item Our assumptions on $\mu$ ensure that the effective
  dissipation rate $\mu (U')^2$ is always smaller than one. The dependence in terms of a
  third root hence reflects the enhancement of the mixing rate due to shear.
\item The nonlinear constant viscosity Navier-Stokes problem has been studied in
  \cite{bedrossian2016sobolev}. This article extends these results in the
  linearized case to the stratified
  viscosity problem. In particular, we extend the by now common
  Cauchy-Kowalewskaya approach to the setting where $U'(y)$ and $\mu(y)$ may
  vary by many orders of magnitude (but may do so only gradually). We expect these methods to be
  of interest of their own for the wider community and applicable also to
  other related problems (e.g. the variable viscosity Boussinesq equations). 
\item Unlike in the constant viscosity setting for the shear flow considered in
  this article the second derivative of the shear $U''$ is non-trivial and does
  not approach zero under the (variable viscosity) heat flow. Hence, as in the inviscid setting
  \cite{Zill3} we require a smallness condition to control the correction term
  $U''v_2$ in the linearized equation (this condition further allows us to
  control derivatives of the viscosity). This motivates our assumption
  \eqref{eq:gradual} (see Section \ref{sec:lemmas} for further discussion).
\item We remark that in view of \eqref{eq:hydrobalance} the shear flow $U$ is
  strictly monotone and hence invertible (but $U'$ might be very large). In our
  analysis it will prove advantageous to thus equivalently consider the variable
  $z=U(y)$ and the profile $W$ moving with the shear. Moreover, stability is
  most naturally phrased in terms of spaces $L^2(dz) \simeq L^2(U'dy)$.
\item The first condition in \eqref{eq:gradual} allows $\mu$ to grow
  exponentially but only with a small exponent. In particular, this implies that
  level sets of the form $\{y: 10^{j}< \mu(y) < 10^{j+1}\}$ for $j\in \Z$ are
  bounded below in size, which we exploit in a partitioning construction in
  Section \ref{sec:glue}.
  The second condition in \eqref{eq:gradual} should be understood as a
  regularity assumption on the relative change $\frac{\mu}{\sup(\mu)}$, which
  should decay at very large frequencies (larger than $\sup(\mu)^{-1/3}$).
\end{itemize}

As we discuss in the following corollary the weighted dissipation estimate implies
exponential decay, if $\omega$ remains suitably localized. 
\begin{cor}
  Let $W(t,x,U(y)):=\omega(t,x-tU(y),y)$ be as in Theorem \ref{thm:main}, let $M \subset \R$ and suppose that on a given time interval $(t_1,t_2)\subset (0,\infty)$ a fraction at least $\theta \in (0,1)$
  of the $L^2$ energy is localized in $M$. That is,
  \begin{align*}
    \|W(t)\|_{L^2(M)}^2 \geq \theta \|W(t)\|_{L^2}^2.
  \end{align*}
  Then for all $t \in (t_1,t_2)$ it holds that
  \begin{align*}
    \|W(t)\|_{L^2}^2\leq \exp(-0.001 \theta \min_{M}\sqrt[3]{\mu (U')^2}(t-t_1))\|W(t_1)\|_{L^2}^2
  \end{align*}
\end{cor}

\begin{proof}
  By Theorem \ref{thm:main} it holds that
  \begin{align*}
    \dt \|W(t)\|_{L^2}^2 \leq -0.01 \| c W\|_{L^2}^2
  \end{align*}
  with $c:=(\mu (U')^2)^{1/6}$.
  By assumption the left-hand-side can be bounded from above by
  \begin{align*}
    - 0.001 \min_{M} c \ \theta \|A W(t)\|^2,
  \end{align*}
  which yields the result.
\end{proof}

We stress that the time interval considered in this corollary might be very small
if $M$ is a region with a very fast effective decay rate, since then the $L^2$
energy (or enstrophy) in that region can be expected to decay much faster than in other regions
and hence will correspond to a much smaller fraction than $\theta$ after a time.

Based on the local dissipation rate at first sight one might also conjecture an
estimate of the form 
\begin{align*}
  \|\exp(\sqrt[3]{\mu (U')^2}t) W(t)\|_{L^2}\leq C \|W(0)\|_{L^2}.
\end{align*}
to hold.
However, such an estimate cannot be expected to hold in general, since the
Biot-Savart law is non-local and not decaying quickly enough.
More precisely, if $W$ is highly localized in a region $M$,
then the velocity field generated by $W$ exhibits decay away from $M$ in terms
of a power law of the distance $\dist(y, M)$. Hence, supposing for the moment
that $W$ remains localized and that $M$ is a region with small decay rate, we
expect $W$ to decay with slower rate. In particular, if $M'$ is a different
region with much higher damping rate, then the decay of the Biot-Savart law in
terms of $\dist(M, M')$ is not sufficiently strong to compensate
for the difference in dissipation rates.

The remainder of our article is structured as follows:
\begin{itemize}
\item In Section \ref{sec:notation} we introduce function spaces, changes of
  variables and notational conventions used throughout the article.
\item As a first model setting in Section \ref{sec:model2} we establish linear $L^2$
  stability for the case when $\mu$ varies only by a bounded factor. This allows
  us to more clearly present the main tools of our proofs and discuss the
  necessity of assumptions.
\item In Section \ref{sec:glue} we extend these $L^2$ stability results to the
  general setting by constructing local versions of several estimates. Here the
  non-local structure of the Biot-Savart law and the interaction of the
  localization and dissipation require careful analysis.
\item Using the linear $L^2$ stability results as a building block, in Section
  \ref{sec:proofmain} we establish linear stability in $H^N$ and thus prove
  Theorem \ref{thm:main}.
\end{itemize}

\section{Stationary Solutions and Notation}
\label{sec:notation}

In this section we establish that the shear flow $U$ given in Theorem
\ref{thm:main} indeed is a stationary solution. Furthermore, we derive the
linearized equation around this state in vorticity formulation.

In our analysis of the Navier-Stokes equations it is often convenient to work in Lagrangian coordinates moving with the underlying shear
flow $(U(y),0)$.
Moreover, since we assume that $U$ is strictly monotone there exists a change of
coordinates $y \mapsto z=U(y)$ which straightens out the flow lines.
For easier reference the equivalent formulations of the equations with respect
to these coordinates are also collected in this section. Moreover, we define
Sobolev spaces and multipliers with respect to $z$.

We remark already here that this construction requires further refinement for the general
situation, but provides a good description if one additionally assumes that $\mu$ is globally
comparable to a constant, which is the model setting of Section
\ref{sec:model2}.
In Section \ref{sec:glue} we replace this global change of variables by a family
of suitably localized coordinate changes, which accounts for the fact that $\mu$
and hence $\p_y U$ may change by many orders of magnitude.

\begin{lem}[Stationary solution]
  \label{lem:stationary}
  Let $\mu \in C^2(\R)$ be a given function with $\mu>0$, let $C \in \R\setminus\{0\}$ and define
  \begin{align*}
    U(y) = \frac{C}{\mu(y)}.
  \end{align*}
  Then $v(x,y)=(U(y),0) \in C^2(\R^2; \R^2)$ is a stationary solution of the
  Navier-Stokes equations with viscosity $\mu$.

  The linearized equations in vorticity formulation around this solution are
  given by
  \begin{align}
    \label{eq:1}
    \begin{split}
    \dt \omega + U(y) \p_x \omega - U'' v_2 &= \dv(\mu \nabla \omega) - \dv(\mu' \nabla v_1) - \mu'' \p_x v_2, \\
    v & =\nabla^{\bot} \Delta^{-1} \omega
  \end{split}
  \end{align} 
\end{lem}

\begin{proof}[Proof of Lemma \ref{lem:stationary}]
  Following Theorem \ref{thm:main} we make the ansatz
 \begin{align}\label{shear}
\mu=\mu(y),\quad v=\begin{pmatrix} U(y)\\ 0\end{pmatrix}.
\end{align}
The Navier-Stokes equations \eqref{NS} then reduce to the following equations
\begin{align*}
\begin{pmatrix}
-\d_{y}(\mu\d_{y} U)+\d_{x}p
\\
\d_{y}p\end{pmatrix}=\begin{pmatrix}0\\0\end{pmatrix}.
  \end{align*}
The second equation $\d_y p=0$ implies $p=P(x)$ for some function $P$ depending
only on $x$, while $\p_y (\mu \p_y U)$ depends only on $y$. Hence, both
functions need to equal a common constant, which yields the hydrostatic balance relation
\begin{align}\label{balance}
\d_y(\mu\d_y U)=C_0
\end{align}
 and $p=P(x)=C_0x+C_1$, where $C_0, C_1\in\R$ are constants.
In particular, specializing to the case $C_0=0$, we verify that our choice of
$U$ yields a stationary solution.

If we also allow for $C_0$ to be possibly non-trivial there are many solutions
of potential interest:
\begin{itemize}
\item The Uniform flow: $U=\text{const}$.
\item The Couette flow: $U=y$, with $\mu=\text{const.}$ or $\mu=C_0y+C_2$.
\item The Poiseuille flow: $U=y(1-y)$, with $\mu=\text{const.}$, $y\in [0,1]$.
\item The shear layer: $U=\tanh(y)$, with $\mu=\sech^{-2}(y)$.
\item The jet or wake: $U=\sech^2(y)$, with $\mu=y\cosh^2(y)\coth(y)$.
\end{itemize}
In this article we restrict to the case $C_0=0$ since then for non-vanishing
viscosity the (non-trivial) shear flow $U$ has no critical points, which would
pose an obstacle to damping estimates.
Furthermore, in view of physical applications we additionally assume that the effective
damping rate $\mu(\p_yU)^2$ is not large.

In the following let $U,\mu$ be solutions of \eqref{balance} which hence are
solutions of the Navier-Stokes equations in velocity formulation.
We may then obtain the equation for the vorticity
\begin{align*}
  \omega=\nabla^\perp\cdot v,\quad \hbox{ with }\nabla^\perp =\begin{pmatrix}-\d_y\\ \d_x\end{pmatrix},
\end{align*}
by applying the operator $\nabla^\perp\cdot$ to the velocity equation \eqref{NS}.
Notice that
\begin{align*}
&\dv(\mu Sv)=\begin{pmatrix}2\d_x \mu \d_x v_1& \d_y\mu\d_y v_1+\d_y\mu\d_x v_2\\
\d_x\mu\d_x v_1+\d_x\mu\d_x v_2& 2\d_y\mu\d_y v_2\end{pmatrix},
\\
&  v=\nabla^\perp\Delta^{-1}\omega=\begin{pmatrix}-\d_y\Delta^{-1}\omega\\ \d_x\Delta^{-1}\omega\end{pmatrix}.
\end{align*}
We may calculate(see also \cite{HL})
\begin{align*}
\nabla^\perp\cdot\dv(\mu Sv)
&=[(\d_{yy}-\d_{xx}) \mu (\d_{yy}-\d_{xx}) +(2\d_{xy})\mu (2\d_{xy} )]\Delta^{-1}\omega 
\end{align*}
which can be equivalently expressed as
\begin{align*}
\Delta(\mu\omega)-2\mu''\d_xv_2=\dv(\mu\nabla \omega)-\dv(\mu'\nabla v_1)-\mu''\d_xv_2.
\end{align*} 
Thus we arrive at the vorticity formulation for the Navier-Stokes equations with
viscosity $\mu$:
\begin{align}\label{vorticity}
\d_t\omega+v\cdot\nabla\omega=\Delta(\mu\omega)-2\mu''\d_xv_2\equiv\dv(\mu\nabla \omega)-\dv(\mu'\nabla v_1)-\mu''\d_xv_2.
\end{align}

Finally, we linearize the vorticity equation \eqref{vorticity} around this shear flow to arrive at the following linearized equation
 \begin{align}\label{omega}
\d_t\omega+U(y)\d_x\omega-U''(y)v_2=\Delta(\mu\omega)-2\mu''\d_xv_2\equiv\dv(\mu\nabla \omega)-\dv(\mu'\nabla v_1)-\mu''\d_xv_2 .
 \end{align}
 
\end{proof}

In the following we introduce some equivalent reformulations of linearized equations \eqref{eq:1} in order to simplify our notation.
We first observe that in the equations \eqref{eq:1}
  \begin{align*}
    \begin{split}
    \dt \omega + U(y) \p_x \omega - U'' v_2 &= \dv(\mu \nabla \omega) - \dv(\mu' \nabla v_1) - \mu'' \p_x v_2, \\
    v & =\nabla^{\bot} \Delta^{-1} \omega
  \end{split}
  \end{align*} 
  all coefficient functions do not depend on $x$ explicitly.
  
  Hence the evolution of the $x$-average of the vorticity which we denote by $\omega_{=}$
  decouples as
  \begin{align*}
    \dt \omega_=&= \p_y (\mu \p_y \omega_{=}) + \p_y (\mu'\omega_{=})=\p_{yy}(\mu\omega_=). 
  \end{align*}
  The $x$ average hence evolves as in a variable coefficient heat equation and does not
  influence the evolution of the orthogonal complement
  \begin{align*}
    \omega_{\neq}=\omega- \omega_{=}.
  \end{align*}
  For this reason we in the following without loss of generality assume that
  initially
  \begin{align*}
    \omega_{=}=0, 
  \end{align*}
   which then remains the case for all times.

  As another consequence of the lack of $x$-dependence, the equations decouple
  after a Fourier transform in $x$, which we denote by
  \begin{align*}
    \hat\omega(t,k,y)=\frac{1}{2\pi}\int e^{-ikx}\omega(t,x,y)dx.
  \end{align*}
  Our equations read:
  \begin{align*}
    \d_t \hat\omega+ikU(y)\hat\omega-U''(y)\frac{ik}{-k^2+\d_{yy}}\hat\omega
    =(-k^2+\d_{yy})(\mu\hat\omega)+2\mu''\frac{k^2}{-k^2+\d_{yy}}\hat\omega.
  \end{align*}
  We  may further consider the vorticity moving with the underlying shear 
\begin{align*}
  \widetilde W(t,x,y)=\omega(t,x+tU(y),y).
\end{align*}
Expressed in Fourier variables it holds that
\begin{align*}
  \mathcal{F}_x \widetilde W(t,k,y)=e^{iktU(y)}\hat\omega(t,k,y),
\end{align*}
and hence
\begin{align}\label{tildeW}
  \begin{split}
&\d_t (\mathcal{F}_x\widetilde W)-\frac{ikU''(y)}{-k^2+(\d_y-iktU'(y))^2}(\mathcal{F}_x\widetilde W)
\\
& =(-k^2+(\d_y-iktU'(y))^2)(\mu\mathcal{F}_x\widetilde W)  +\frac{2\mu''k^2}{-k^2+(\d_y-iktU'(y))^2}(\mathcal{F}_x\widetilde W).
\end{split}
\end{align}

As a final step we observe that by assumption $U(y)$ is strictly monotone and
hence there exists a change of variables $y \mapsto z=U(y)$, which serves as our
main formulation in the following.

\begin{lem}[Vorticity formulation]
\label{lem:z}
  Let $U, \mu$ be as in Theorem \ref{thm:main} and consider the linearized
  equations in vorticity formulation.
  Further denote the change of coordinates $y \mapsto z=U(y)$ and define
  \begin{align*}
    W(t,x,z)= \omega(t,x+tU(y),y).
  \end{align*}
  Then $\omega$ solves \eqref{eq:1} if and only if for every $k \in \Z$ the
  Fourier transform of $W$ with respect to $x$ solves:
  \begin{align}
    \begin{split}
  &  \d_t (\mathcal{F}_xW)-\frac{ikU''}{-k^2+(U'(\d_z-ikt))^2}(\mathcal{F}_xW)
  \\
  &=(-k^2+(U'(\d_z-ikt))^2)(\mu\mathcal{F}_xW) +\frac{2\mu''k^2}{-k^2+(U'(\d_z-ikt))^2}(\mathcal{F}_xW),
\end{split}
    \end{align}
    where with slight abuse of notation coefficient functions are evaluated in
    $y$ such that $z=U(y)$, e.g. $U''= \p_y^2U|_{y=U^{-1}(z)}$.
    In other words, by introducing
    \begin{equation}\label{nablat}
    \nabla_t:=\begin{pmatrix} \d_x\\ U' (\d_z-t\d_x)\end{pmatrix},
    \end{equation}
    we may write the above equation for $\mathcal{F}_xW$ as
  \begin{align}\label{Equation:W}
    \begin{split}
  &  \d_t W-U''   V_2
  =\dv_t (\mu \nabla_t  W) 
  -\dv_t(\mu' \nabla_t   V_1)
  -\mu'' \d_x V_2.
\end{split}
    \end{align}
    where
    \begin{align*}
      V_1&=\frac{-U'(\d_z-t\d_x)}{\d_x^2+(U'(\d_z-t\d_x))^2}  W, \\
    \quad   V_2&=\frac{\d_x}{\d_x^2+(U'(\d_z-t\d_x))^2} W.
    \end{align*}

\end{lem}

\begin{proof}[Proof of Lemma \ref{lem:z}]
  Since by assumption
  \begin{align*}
    \mu \p_y U
  \end{align*}
  equals a non-trivial constant and $\mu>0$ does not vanish, it follows that $U$
  is strictly monotone and hence invertible.
  Thus, the above claimed change of variables exists.
  Furthermore, it holds that
  \begin{align*}
    \p_y -t U'\p_x= U' \p_z- t U'\p_x = U' (\p_z-t\p_x),
  \end{align*}
  which together with \eqref{tildeW} concludes the proof.
\end{proof}
In the following sections we establish asymptotic stability of $W$ in Sobolev
regularity.
More precisely, we will first consider the special case where $U$ is globally
bilipschitz with comparable upper and lower Lipschitz constants in Section
\ref{sec:model2}.
Building on these results, in Section \ref{sec:glue} we consider the general
case, where we further introduce modified changes of coordinates adapted to the
local behavior of the coefficient functions. Finally, in Section
\ref{sec:proofmain} we bootstrap the stability results in $L^2$ to establish
stability in $H^N$.

Unless noted otherwise we here always work in coordinates with respect to $z$ and
may without loss of generality assume that $W$ is localized at frequency $k\neq 0$,
arbitrary but fixed, with
respect to $x$.
We thus briefly write
\begin{align*}
  L^2 := L^2(\R, dz)
\end{align*}
and use
\begin{align*}
  \langle \cdot,\cdot \rangle
\end{align*}
to refer to the inner product on that space.

\section{A Model Case and $L^2$ Estimates}
\label{sec:model2}
In this section we consider a special case of the linearized Navier-Stokes
equations \eqref{eq:1} in vorticity formulation
\begin{align*}
  \dt \omega + U(y) \p_x \omega - U'' v_2 &= \dv(\mu \nabla \omega) - \dv(\mu \nabla v_1) - \mu'' \p_x v_2,\\
  v&= \nabla^{\bot}\Delta^{-1}\omega,
\end{align*}
or rather
\begin{align*}
   \d_t W-U''   V_2
  &=\dv_t (\mu \nabla_t  W) 
  -\dv_t(\mu' \nabla_t   V_1)
  -\mu'' \d_x V_2.\\
  W(t,x,U(y))&=\omega(t,x-tU(y),y),
\end{align*}
in which we additionally assume that $\mu$ is comparable to a constant globally instead of just locally.
More precisely, in this section we additionally require that
\begin{align*}
  \frac{\sup(\mu)}{\inf(\mu)} \leq 100.
\end{align*} 
We note that this further implies that $U$ is bilipschitz and hence allows for a
global change of variables to $z=U(y)$ (see Section \ref{sec:notation} and Lemma
\ref{lem:z}).
This simplification therefore allows us to more clearly present commutator estimates and
introduce techniques of proof.

In a second step in Section \ref{sec:glue} we use that such bounds are true
\emph{locally}, that is when $\mu$ and $U$ are restricted to bounded intervals
of a given length.
Extending these restrictions to functions on the whole space which are bounded
above and below we will hence be able to use this section's results for the
``localized'' problems. A key challenge then lies in controlling non-linear
interaction due to the Biot-Savart law and in ``gluing'' the various estimates
in a way that preserves dissipation and decay estimates.

The following proposition summarizes our main results for this section and
employs a by now common Lyapunov functional/energy approach (see for instance
\cite{masmoudi2020stability,bedrossian2016enhanced,tao20192d,liss2020sobolev}), where a key challenge lies in constructing a suitable time, frequency and
space-dependent operator $A$ which captures possible growth in the evolution of
solutions to \eqref{eq:1}.

\begin{prop}
\label{prop:model}
  Let $\mu, U$ satisfy the assumptions of Theorem \ref{thm:main} and
  additionally suppose that
  \begin{align*}
    \frac{\sup(\mu)}{\inf(\mu)} \leq 100.
  \end{align*}
  Then there exists a
  time-dependent family of operators $A(t)$ such that for any initial data
  $\omega_0\in L^2$ it holds that
  \begin{align*}
  c \|W(t)\|_{L^2} \leq \|A(t)W(t)\|_{L^2} &\leq \|W(t)\|_{L^2}.
  \end{align*}
  Furthermore, define $u=\inf U'$ and $\nu:=\inf \mu (U')^2$ (note that $\nu=u
  \mu(0)U'(0)$, since $\mu U'$ is constant) then it holds that 
  \begin{align*}
    \frac{d}{dt} \|AW \|_{L^2}^2 \leq -0.001 \|\mathcal{F}(\nu^{1/3}+ \nu (\xi-kt)^2 + \frac{1}{1+u^2(\xi-kt)^2})\mathcal{F}^{-1}AW\|_{L^2}^2 \leq 0 
  \end{align*}
  In particular, the linear stability estimates in $L^2$ of Theorem
  \ref{thm:main} hold in this setting.
\end{prop}

\begin{itemize}
\item The operator $A(t)$ is defined in terms of a Fourier multiplier in
  Definition \ref{defi:multiplier}.   
\item 
We remark that in the present setting by assumption $\mu$ and hence $U'$ may
only vary by a factor $100$. Therefore $u$ is also comparable to an average
value of $U'$.
\item 
The decay by $\nu^{1/3}+ \nu (\xi-kt)^2$ quantifies the enhanced dissipation
mechanism. More precisely, if $|\xi-kt|\geq \nu^{-1/3}$ the latter term
dominates, but for frequencies with $|\xi-kt|$ smaller than this the enhanced
rate $\nu^{1/3}$ still persists (see Definition \ref{defi:multiplier} and Lemma \ref{lem:dissipationterm} for further
discussion).
\item 
The last multiplier $\frac{u}{1+u^2(\xi-kt)^2}$ corresponds to control of the
velocity perturbation. Since $u$ is bounded below this contribution is dominated
by $\nu^{1/3}$ for large frequencies. The statement of the theorem thus is that
even for frequencies where $|\xi-kt|$ is small this control persists. In
particular, this control remains valid as $\nu$ tends to zero, which is crucial
for the control of $U'' v_2$ in the evolution equation (see Lemma
\ref{lem:velocityerrors}).
\item We note that here the choice of $G$ involves a small factor $0.1$, as do
  the dissipation rates in Theorem \ref{thm:main}. These small factors allow use
  some flexibility in the control of interaction terms (see Section \ref{sec:lemmas}).
\end{itemize}

In the interest of clear presentation we formulate the main steps of the proof
as a series of lemmas. We then show how to use these to establish Proposition
\ref{prop:model} before proving the lemmas at the end of this section.

 \begin{defi}[Decreasing Multiplier and Fourier sets]
   \label{defi:multiplier}
   Let $\mu, U$ be a given stationary solution as in Theorem \ref{thm:main} and
   define the \emph{local dissipation rate} $\nu$ as
   \begin{align*}
     \nu: = \inf \mu (\p_y U)^2.
   \end{align*}
   and the local shear rate as
   \begin{align*}
     u:=\inf |U'|.
   \end{align*}
   Let further
   \begin{align*}
     G:= 0.1 \nu^{-1/3}.
   \end{align*}  
   We then define the \emph{good set} $G_t \subset \Z \times \R$ by
   \begin{align*}
     G_t:= \{(k,\xi): k \neq 0, |\frac{\xi}{k}-t|\geq G\},
   \end{align*}
   and the \emph{bad} set $B_t$ as the complement (excluding $k=0$)
   \begin{align*}
     B_t=\{(k,\xi) \in \Z \times \R: k \neq 0, |\frac{\xi}{k}-t|< G\}.
   \end{align*}
   For any fixed $k$, if $B_t \cap \{k\}\times \R$ is non-empty, the set $G_t
   \cap \{k\}\times \R$ has two connected components, where we denote by
   $G_t^{-}$ the half-line extending to $-\infty$ and by $G_t^{+}$ the half-line
   extending to  $+\infty$.

   Associated with this partition we define a Fourier multiplier $m$ by
   \begin{align*}
     \dt m(t,k,\xi) =
     \begin{cases}
       m(t,k,\xi)(-\nu^{1/3}- \frac{u}{1+u^2(\frac{\xi}{k}-t)^2}), & \text{ if } (k,\xi) \in B_t, \\
       0, & \text{else.}
     \end{cases}
   \end{align*}
   and the asymptotic condition $\lim_{t\rightarrow -\infty}m(t,k,\xi)=1$.

   We denote the operator associated with the Fourier multiplier $m$ by $A(t)$:
   \begin{align*}
    A \phi = \mathcal{F}^{-1}m \mathcal{F}\phi,
   \end{align*}
   where $\mathcal{F}$ denotes the Fourier transform with respect to $x$ and $z=U(y)$.
 \end{defi}

This multiplier combines features of the inviscid multiplier of \cite{Zill3} and
the constant viscosity multiplier of \cite{liss2020sobolev,bedrossian2016sobolev}. 
\begin{itemize}
\item The relative decay of $\mu$ by $-\nu^{1/3}$ compensates for the relatively
  weak dissipation in the \emph{bad} Fourier region. Here the decay of $A$ allows
  to establish damping of $AW$.
\item The second multiplier models the growth of $v_2$ as given by the
  Biot-Savart law or rather of $U' v_2$. As we discuss in the proof of Lemma
  \ref{lem:velocityerrors} we then use that by assumption $U'' =
  \frac{U''}{U'}U'$ is small compared to $U'$ and hence the linearization error
  $U'' v_2$ can be controlled by this multiplier when we are in the \emph{bad} region.  
\end{itemize}

As we prove in the following subsection the multiplier $m$ (and hence the
operator $A$) satisfies several useful bounds and, in particular, serves to
control various error terms when $W$ is concentrated in the \emph{bad} set. 
 \begin{lem}
\label{lem:propertiesofm}
   Let $m$ be as in Definition \ref{defi:multiplier}. Then $m$ satisfies the
   following estimates:
   \begin{enumerate}
   \item There exists a constant $c$ independent of $\xi$ and $t$ such that
     \begin{align*}
       c \leq m \leq 1.
     \end{align*}
   \item The multiplier $m$ is constant (independent of $\xi$ and $t$, but might
     depend on $k$) for large positive or negative times. By the conventions of
     our definition one of these constants is chosen as $1$ and the other as $c$:
     \begin{align*}
       m(t,k,\xi)&=c \text{ if } t>\frac{\xi}{k}+G,\\
       m(t,k,\xi)&=1 \text{ if } t<\frac{\xi}{k}-G.
     \end{align*}
   \item The operator $A$ is a continuous invertible operator from $L^2$ to
     $L^2$ and satisfies
     \begin{align*}
       c \|\phi\|_{L^2} \leq \|A\phi\|_{L^2} \leq \|\phi\|_{L^2}
     \end{align*}
     for all $\phi \in L^2(dz)$.
   \end{enumerate}
 \end{lem}

 Given this definition of our multiplier our main task in the following is to establish suitable
 estimates for
 \begin{align*}
   \frac{d}{dt}\|AW\|_{L^2}^2/2 &= \langle AW, \dot{A} W \rangle + \langle AW, A \dt W \rangle\\
   &=  \langle AW, \dot{A} W \rangle + \langle A U'' V_2, A W  \rangle\\
   & \quad + \langle AW, A\dv_t(\mu \nabla_t)W \rangle\\
   &\quad - \langle AW, A\dv_t(\mu' \nabla_t)V_1   \rangle \\
   &\quad - \langle AW, A \mu'' \p_x V_2 \rangle,
 \end{align*}
 where we used the equation \eqref{Equation:W} to rewrite $A\dt W$.
More precisely, we intend to show that the dissipation and the decay of $m(t)$
are strong enough to absorb possible growth and that hence $\|AW\|^2$ is
decreasing in time.
Integrating these estimates we thus obtain a Lyapunov functions, which allows us to prove Proposition
 \ref{prop:model}.

The following lemma quantifies the combined strength of the dissipation
mechanism and the decay of the multiplier. 
 \begin{lem}[The dissipation term]
   \label{lem:dissipationterm}
Let $t\geq 0$, let $A, G$ and $m$ be given by Definition~\ref{defi:multiplier} and
let $W \in L^2$ be a given function.
Then it holds that
\begin{align*}
  & \quad 0.2 \langle AW, \dot{A} W \rangle + \langle AW, A\dv_t(\mu \nabla_t)W \rangle 
  \\
  &\leq -0.1 \Bigl\|\mathcal{F}^{-1} \Bigl( \nu^{1/3} + u^2 (\xi-kt)^2 + \frac{u}{1+u^2(\frac{\xi}{k}-t)^2}\Bigr) \mathcal{F} AW\Bigr\|^2,
\end{align*}
where $u$ and $\nu$ are defined as in Definition~\ref{defi:multiplier}.
\end{lem}
This decay lies at the core of our damping mechanism.
In the following lemmas we show that all other contributions to
$\frac{d}{dt}\|AW\|^2$ can be considered as errors.

We begin by the errors in the dissipation term due to the fact $\mu$ is
non-constant.
Here we recall that by the assumptions imposed in Theorem \ref{thm:main} the
relative change of $\mu$ is required to be small:
\begin{align*}
 |\frac{\mu'}{\mu}| \leq 0.001
\end{align*}
as well as
\begin{align*}
  |\p_z \frac{\mu'}{\mu}| \leq 0.001. 
\end{align*}
This smallness together with the fact that commutator terms involve
integro-differential operators of lower order than $2$ allows us to control
these errors.
 \begin{lem}[Viscosity errors]
\label{lem:viscosityerrors}
Let $t\geq 0$, let $A, C, \nu, u$ and $m$ be given by Definition~\ref{defi:multiplier} and
let $W \in L^2$ be a given function.
Then it holds that
\begin{align*}
  & \quad 0.2 \langle AW, \dot{A} W \rangle - \langle AW, A(\dv_t(\mu' \nabla_t)v_1+ \mu'' \p_x v_2)  \rangle \\
  &\leq 0.01 \|\mathcal{F}^{-1}\left(   \nu^{1/3} + u^2 (\xi-kt)^2 + \frac{u}{1+u^2(\frac{\xi}{k}-t)^2}\right) \mathcal{F} AW\|^2.
\end{align*}
\end{lem}

\begin{lem}[Velocity errors]
  \label{lem:velocityerrors}
   Let $t\geq 0$, let $A, C, \nu, u$ and $m$ be given by Definition~\ref{defi:multiplier} and
let $W \in L^2$ be a given function.
Then it holds that
\begin{align*}
 &\quad  0.2 \langle AW, \dot{A} W \rangle - \langle AW, A(U'' v_2)  \rangle \\
  &\leq  0.01 \|\mathcal{F}^{-1}\left( \nu^{1/3} + u^2 (\xi-kt)^2 + \frac{u}{1+u^2(\frac{\xi}{k}-t)^2}\right) \mathcal{F} AW\|^2.
\end{align*}
\end{lem}

The proofs of Lemmas \ref{lem:propertiesofm} to \ref{lem:velocityerrors} are
given in the following Section \ref{sec:lemmas}. We briefly discuss how to
combine the estimates of the the lemmas to
establish Proposition \ref{prop:model}.

\begin{proof}[Proof of Proposition \ref{prop:model}]
  Let $A, m$ be given as in Definition \ref{defi:multiplier}, let $W$ denote the
  solution of the linearized Navier-Stokes equations and consider the energy
  \begin{align*}
    E(t)= \|A(t) W(t)\|_{L^2}^2.
  \end{align*}
  Then by the results of Lemmas \ref{lem:dissipationterm},
  \ref{lem:viscosityerrors} and \ref{lem:velocityerrors} it follows that
  \begin{align}
    \label{eq:decay}
    \dt E \leq -0.05 \|\mathcal{F}^{-1} \nu^{1/3} + u^2 (\xi-kt)^2 + \frac{u}{1+u^2(\frac{\xi}{k}-t)^2} \mathcal{F} AW\|^2.
  \end{align}
  In order to conclude we recall that by Lemma \ref{lem:propertiesofm} the
  Fourier multiplier $m$ corresponding to $A$ is bounded between $c$ and $1$ and
  we may hence relate $E$ with $\|W\|_{L^2}^2$.
\end{proof}

In Section \ref{sec:glue} we will extend these estimates to the case where $\mu$
(and hence $U'$) is allowed to vary by many orders of magnitude.
In particular, the values of $\nu$ and $\mu$ then are only locally defined.
A key challenge there is to show that non-local effects and interactions between
different regions in space can be controlled in a sufficiently good way.
Finally, in Section \ref{sec:proofmain} we show that the damping estimates in
$L^2$ can be bootstrapped to yield stability in arbitrary Sobolev regularity,
following an argument of \cite{zillinger2021linear}.

\subsection{Proof of Lemmas}
\label{sec:lemmas}

We begin by discussing the properties of the multiplier $m$, which may be computed
explicitly in terms of integrals.
 \begin{proof}[Proof of Lemma \ref{lem:propertiesofm}] 
   By definition it holds that $\dt m \leq 0$ and hence
   \begin{align*}
     m(t) = m(-\infty) + \int_{-\infty}^t \dt m \leq m(-\infty)=1.
   \end{align*}
   Furthermore, we may explicitly compute $m$ as
   \begin{align}
     \label{eq:mexplicit}
     m(t,k,\xi)= \exp \left(\int_{-\infty}^t \frac{\dt m}{m} \,1_{B_\tau}(k,\xi) \,d\tau\right),
   \end{align}
   where we used that $m(-\infty,k,\xi)=1$.
   It then holds that
   \begin{align*}
     \int_{-G}^G \nu^{1/3} dt = 2 G \nu^{1/3} \leq 0.2
   \end{align*}
   is bounded by a uniform constant (and further improves for $k$ large).
   Furthermore, also
   \begin{align*}
     \int_{-G}^G \frac{u}{1+u^2 t^2} dt = \arctan(\tau)|_{\tau=-uG}^{uG} \leq \pi
   \end{align*}
   is uniformly controlled.
   Therefore, we may estimate
   \begin{align*}
     m(t,k,\xi) \geq \exp(-0.2-\pi).
   \end{align*}

   We further observe that for $t< \frac{\xi}{k}-G$ or $t>\frac{\xi}{k}-G$ it
   holds that
   \begin{align*}
     \frac{\dt m}{m}=0
   \end{align*}
   and $m$ is thus constant on these intervals.
   On the left interval $m$ equals $m(-\infty,k,\xi)=1$, while on the right it
   equals
   \begin{align*}
     \exp(-\int_{-G}^G\nu^{1/3}+ \frac{u}{1+u^2 t^2} dt):=c.
   \end{align*}
   Finally, by Parseval's identity these bounds for the multiplier $m$ are
   equivalent to $L^2$ bounds for the operator $A$.
 \end{proof}

Having established these properties of the operator $A$ we next turn to
estimating the dissipation.
Here for \emph{good} frequencies in the sense of Definition
\ref{defi:multiplier} the dissipation is rather strong.
If one instead considers \emph{bad} (that is, close to resonant) frequencies we rely on
the fact that $m$ decreases in time $t$ to provide sufficient decay.

 \begin{proof}[Proof of Lemma \ref{lem:dissipationterm}]
We recall that by Lemma \ref{lem:propertiesofm} the multiplier $m$ is multiple
of the identity on the connected components of the \emph{good} Fourier set
$G_t$.
In particular, when restricted to these sets $A$ commutes with all other operators.
In our proof we hence expand
\begin{align*}
  W = \mathcal{F}^{-1}1_{G_t^{-}}\mathcal{F}W + \mathcal{F}^{-1}1_{B_t} \mathcal{F}W +\mathcal{F}^{-1} 1_{G_t^{+}}\mathcal{F}W =: W_1 + W_2+ W_3
\end{align*}
according to this Fourier decomposition.
That is, we study
\begin{align*}
  \langle AW_i, A\dv_t(\mu \nabla_t)W_j\rangle
\end{align*}
with $i,j \in \{1,2,3\}$.

We begin by discussing the diagonal cases $i=j$ in the \emph{good} regime.

\noindent{\textbf{Estimates for $ \langle AW_j, A\dv_t(\mu \nabla_t)W_j\rangle$, $j=1$ or $3$:}}\\
We recall that $A=\Id$ or $c\ \Id$ when applied to $W_1$ or $W_3$, respectively.
Hence, we may explicitly compute that
\begin{align*}
 & \langle AW_1, A\dv_t(\mu \nabla_t)W_{1} \rangle = \langle W_1, \dv_t(\mu \nabla_t)W_{1} \rangle \\
 & = \langle W_1, \mu \p_x^2 + U' (\p_z-t\p_x) \mu U' (\p_z-t\p_x) W_1\rangle \\
 & = - \langle \p_x W_1, \mu \p_x W_1\rangle - \langle U'(\p_z-t\p_x)W_1, \mu U' (\p_z-t\p_x) W_1 \rangle 
 - \langle (\p_z U' W_1),  \mu U' (\p_z-t\p_x) W_1  \rangle \\
 & \leq -\nu \|(\p_z-t\p_x) W_1\|^2 -  \langle (\p_z U') W_1,  \mu U' (\p_z-t\p_x) W_1  \rangle,
\end{align*}
where we estimated the first summand by $0$ from above and used that $\mu
(U')^2\geq \nu$ by definition.
For the last term we recall that $\mu$ and hence $U'$ is slowly varying and we
may therefore control
\begin{align*}
 | \langle (\p_z U') W_1,  \mu U' (\p_z-t\p_x) W_1  \rangle |
  &\leq \|\frac{\p_z U'}{U'}\|_{L^\infty} \|U' W_1\|_{L^2} \|\mu U' (\p_z-t\p_x)W_1\|_{L^2} \\
  &\leq \|\frac{\p_z U'}{U'}\|_{L^\infty} 100  G^{-1} \|u (\p_z-t\p_x) W_1\|  \|\mu U' (\p_z-t\p_x) W_1\| \\
  &= \|\frac{\p_z U'}{U'}\|_{L^\infty} 100  G^{-1} \nu \|(\p_z-t\p_x)W_{1}\|_{L^2}^2,
\end{align*}
where we used that $(\p_z-t\p_x)$ is invertible with operator norm of the
inverse map
bounded by $G^{-1}$ on the \emph{good} set and controlled $|U'|\leq 100 u$.
Since $G^{-1}$ is very small, as is $\|\frac{\p_z U'}{U'}\|_{L^\infty}$, it
holds that
\begin{align}\tag{Cond.1}\label{Condition1}
  100 G^{-1} \| \frac{\p_z U'}{U'} \|_{L^\infty} \leq 0.001
\end{align}
and hence this error can be absorbed into the decay.
We thus note that 
\begin{align*}
  \langle AW_1, A\dv_t(\mu \nabla_t)W_{1} \rangle \leq -\|\sqrt{\mu}\d_x AW_1\|_{L^2}^2- 0.9 \nu \|(\p_z-t\p_x)AW_1\|^2
\end{align*}
and by the same argument
\begin{align*}
  \langle AW_3, A\dv_t(\mu \nabla_t)W_{3} \rangle \leq -\|\sqrt{\mu}\d_xA W_3\|_{L^2}^2- 0.9 \nu \|(\p_z-t\p_x) AW_3\|^2.
\end{align*}

\medskip

Since $\mu$ and $U'$ are not constant there further is some non-trivial
interaction between different \emph{good} contributions.

\noindent{\textbf{Estimates for  $\langle AW_i, A\dv_t(\mu \nabla_t)W_j\rangle$
    with $(i,j)=(1,3)$ or $(i,j)=(3,1)$}:}\\
We observe that
\begin{align*}
  \langle AW_{i}, A\dv_t(\mu \nabla_t)W_j\rangle = \langle A^2 W_{i}, \dv_t(\mu \nabla_t)W_j\rangle 
\end{align*}
and that $A^2$ is a multiple of the identity. In the following we may thus for
simplicity of notation instead consider
\begin{align*}
  \langle W_{i}, \dv_t(\mu \nabla_t)W_j\rangle &= - \langle \p_x W_i, \mu \p_x W_j \rangle - \langle U'(\p_z-t\p_x) W_i, \mu U' (\p_z-t\p_x) W_{j} \rangle \\
  &\quad + \langle U' W_i, \frac{\p_z U'}{U'} \mu U' (\p_z-t\p_x) W_{j} \rangle.
\end{align*}
By the same argument as above
\begin{align*}
  |\langle U' W_i, \frac{\p_z U'}{U'} \mu U' (\p_z-t\p_x) W_{j} \rangle| \leq \frac{1}{10} \nu  \|(\p_z-t\p_x) W_{1}\| \|(\p_z-t\p_x)W_{3}\|
\end{align*}
can be considered a negligible error term.

For the remaining terms we instead exploit that $W_i$ and $W_j$ have disjoint
support in Fourier space and that these supports have distance at least $2G$.
In particular, we may estimate
\begin{align*}
  & \quad \langle \p_x W_i, \mu \p_x W_j \rangle 
  = \int \mathcal{F}(\mu)(\xi_1) \mathcal{F}(\p_x W_i)(\xi_2) \overline{\mathcal{F}(\p_x W_i)}(\xi_1+\xi_2) d\xi_1 d\xi_2 \\
  &\leq \|\frac{1}{\xi_1}\mathcal{F}(\p_z\mu)(\xi_1)\|_{L^1(|\xi_1|\geq 2G)} \|\p_x W_1\|_{L^2} \|\p_x W_3\|_{L^2} \\
  &\leq G^{-1} \|\mathcal{F}(\p_z\mu)(\xi_1)\|_{L^1(|\xi_1|\geq G)}\frac{1}{\inf(\mu)} \|\sqrt{\mu} \p_x W_1 \|_{L^2} \|\sqrt{\mu} \p_x W_3 \|_{L^2}. 
\end{align*}
This interaction term can hence absorbed by the decay provided
\begin{align}
  \label{Condition2}
  G^{-1}\|\mathcal{F}(\p_z\mu)(\xi_1)\|_{L^1(|\xi_1|\geq G)}\frac{1}{\inf(\mu)} \leq 0.1,
\end{align}
which is part of our assumption \eqref{eq:gradual} that the relative size of
$\mu$ is slowly varying. We remark that here $\mathcal{F}(\mu)$ refers to the
distributional Fourier transform, since $\mu$ is bounded but not in $L^1$.

By the same argument and using that
\begin{align*}
  \frac{\p_z U'}{U'} = - \frac{\p_z \mu}{\mu}
\end{align*}
we observe that
\begin{align*}
  & \quad |\langle U'(\p_z-t\p_x) W_i, \mu U' (\p_z-t\p_x) W_{j} \rangle| \\
  &\leq \frac{\sup_{|\xi_1|\geq 2 G} \mathcal{F}(U')(\xi_1)}{u} \nu  \|(\p_z-t\p_x) W_{1}\| \|(\p_z-t\p_x)W_{3}\|,
\end{align*}
which can be absorbed provided
\begin{align}
  \frac{\sup_{|\xi_1|\geq 2 G} |\mathcal{F}(U')(\xi_1)|}{\inf(U')} \leq 0.5,
\end{align}
which holds by assumption.

\noindent{\textbf{Estimates for terms involving $W_2$:}}\\
It remains to discuss the influence of the part $W_2$ Fourier-localized in the
\emph{bad set}.

We first study the self-interaction term:
\begin{align*}
  \langle AW_{2}, A\dv_t(\mu \nabla_t)W_{2} \rangle &= \langle AW_2, A \p_x \mu \p_x W_2\rangle \\
                                                    &\quad +\langle A W_2, A (\p_z-t\p_x) U' \mu U' (\p_z-t\p_x) W_2 \rangle \\
  &\quad - \langle A W_2, A (\p_zU') \mu U' (\p_z-t\p_x) W_2 \rangle.
\end{align*}
Since none of $A$, $\mu$ and $U'$ are constant, we cannot easily appeal to the
negativity of the elliptic operator in this regime.
Instead we use that
\begin{align*}
  \langle A W_2, A (\p_z-t\p_x) U' \mu U' (\p_z-t\p_x) W_2 \rangle 
   \leq C \nu G^2 \|AW_2\|_{L^2}^2 
\end{align*}
since $W_2$ is localized in the Fourier set where $|\xi-kt|$ is not large (yet).
We recall that here $G$ was chosen in such a way that $\nu G^2 < 0.01 \nu^{1/3}$ and
thus the dissipation in this \emph{bad} region is weaker than desired.

Similarly, we may estimate
\begin{align}
  \label{eq:dxpart}
  \langle AW_2, A \p_x \mu \p_x W_2\rangle \leq C \|\sqrt{\mu}\p_x A W\|_{L^2}^2 \leq C \nu k^2 \|AW\|_{L^2}^2,
\end{align}
where we used that $(U')^2\geq 1$ and hence $\inf(\mu) \leq \nu$.
As remarked following Definition \ref{defi:multiplier} we may without loss of
generality only consider those $k$ for which
\begin{align*}
  \nu k^2 < 0.001 \nu^{1/3}
\end{align*}
is much smaller than the enhanced dissipation rate, since otherwise this
horizontal dissipation already achieves the desired decay.
With this understanding we restrict to this case for the reminder of the
article.
It then also holds that
\begin{align*}
  \langle AW_2, A \p_x \mu \p_x W_2\rangle \leq C \nu^{1/3} \|AW_2\|_{L^2}^2
\end{align*}
with a small constant $C$.
Furthermore, we may control 
\begin{align*}
  |\langle A W_2, A (\p_zU') \mu U' (\p_z-t\p_x) W_2 \rangle|
   & \leq C \|\frac{\p_zU'}{U'}\|_{L^\infty} \|\mu U' (\p_z-t\p_x)W_2\|_{L^2} \|A\|^2 \\
  & \quad \quad (\|\p_x W_2\|_{L^2} + \|u(\p_z-t\p_x)W_2\|_{L^2}) \\
  & \leq C \nu^{1/3} \|AW_2\|_{L^2}^2
\end{align*}
with a small absolute constant $C$, by our choice of cut-off $G$.

Since the decay of $A$ yields that
\begin{align*}
  \langle A W_2, \dot{A}W_2 \rangle \leq -\nu^{1/3} \|AW_2\|_{L^2}^2
\end{align*}
these contributions can be absorbed.

Finally, it remains to discuss the cross terms
\begin{align*}
  \langle AW_{i}, \dv_t(A \mu \nabla_t)W_{j} \rangle &=  \langle AW_i, A \p_x \mu \p_x W_j\rangle \\
                                                    &\quad +\langle A W_i, A (\p_z-t\p_x) U' \mu U' (\p_z-t\p_x) W_j \rangle \\
  &\quad + \langle A W_i, A (\frac{\p_zU'}{U'}) U' \mu U' (\p_z-t\p_x) W_j \rangle,
\end{align*}
where one $i,j$ equals $2$.
For the first term we estimate by
\begin{align*}
  C \|\sqrt{\mu}\p_x A W_i\|_{L^2} \|\sqrt{\mu}\p_x A W_j\|_{L^2}, 
\end{align*}
which can be absorbed as above.
Similarly, the second term can be controlled by
\begin{align*}
  C \|u(\p_z-t\p_x)W_{1,3}\|_{L^2} \nu^{1/3} \|W_2\|
\end{align*}
and the last term by
\begin{align*}
  C G^{-1}  \|\frac{\p_z U'}{U'}\|_{L^{\infty}} \|u(\p_z-t\p_x)W_{1,3}\|_{L^2} \nu^{1/3} \|W_2\|.
\end{align*}
We may thus use use Young's inequality and the previously obtained damping
estimates to absorb these error terms, which concludes the proof.
 \end{proof}
 
In the previous lemma we have established dissipation due the ``main'' term of
the dissipation operator. The following lemma shows that this decomposition is
indeed and all terms involving higher derivatives of $\mu$ can be considered
lower order.
We recall here that by assumption $\mu$ is slowly varying in the sense that
\begin{align*}
  \|\frac{\mu'}{\mu}\|_{L^\infty}
\end{align*}
and
\begin{align*}
  \|\p_z \frac{\mu'}{\mu}\|_{L^\infty}
\end{align*}
are small.

\begin{proof}[Proof of Lemma \ref{lem:viscosityerrors}]
Given the decay established in Lemma \ref{lem:dissipationterm} we may use
integration by parts, Hölder's inequality and Young's inequality to reduce our
proof to establishing bounds on norms of
\begin{align*}
  \mu' \nabla_tv_1
\end{align*}
and of
\begin{align*}
  \mu'' \p_x v_2.
\end{align*}
In analogy to the proof of Lemma \ref{lem:dissipationterm} we here again
distinguish between the parts of $v_1,v_2$ generated by the vorticity in the
\emph{good} region $W_1, W_3$ and the one localized in the \emph{bad} region
$W_2$.

More precisely, we note that as in \cite{coti2019degenerate} in these estimates
we may replace $v= \nabla_t \phi$, defined in terms of the usual stream
function, by simpler potential using the averaged value $u$ of $U'$.
For this purpose we define $\psi$ to satisfy
\begin{align*}
  (\p_x^2 + u^2 (\p_z-t\p_x)^2)\psi = W = (\p_x^2 + (U'(\p_z-t\p_x))^2)\phi
\end{align*}
with suitable integrability assumptions at infinity.
Then testing these equations with either $\psi$ or $\phi$ one obtains that the energies
\begin{align*}
  \|v\|_{H^N}^2 \approx \|\p_x\psi\|^2 + \|u(\p_y-t\p_x)\psi\|^2
\end{align*}
are comparable (in the sense of bilinear forms acting on $W$).

In the following we may thus discuss
\begin{align*}
  V=(-u(\p_z -t\\p_x), \p_x)\psi
\end{align*}
in place of $v$.

We then observe that in the \emph{good} Fourier region
\begin{align*}
  \mathcal{F}(\p_x V_2) = \frac{k^2}{k^2 + u^2(\xi-kt)^2} \mathcal{F}(W)
\end{align*}
is controlled by $c^{-2} \mathcal{F}(W)$.

Similarly, in the \emph{good} Fourier region
\begin{align*}
  \nabla_t V_1 
\end{align*}
can be controlled by the dissipation, since
\begin{align*}
  \frac{1}{1+u^2(\frac{\xi}{k}-t)^2} \leq \frac{1}{1+G^2}\ll \nu^{1/3}.
\end{align*}

It thus only remains to discuss the \emph{bad} Fourier region. However, there
the weight $A$ is chosen in just such a way that
\begin{align*}
  \|(\p_x, u(\p_z-t\p_x))\psi_{A}\|^2
\end{align*}
can be absorbed using
\begin{align*}
  \langle A W, \dot{A}W \rangle.
\end{align*}
More precisely, we may estimate
\begin{align*}
  u \|(\p_x, u(\p_z-t\p_x))\psi_{A}\|^2 = \int \frac{u}{k^2+u^2(\xi-kt)^2} |AW|^2\\
  \leq \int ( 1_{B_t}(k,\xi) \frac{u}{k^2+u^2(\xi-kt)^2}+0.1 \nu^{1/3}1_{G_t}) |AW|^2,
\end{align*}
since
\begin{align*}
  \frac{u}{k^2+u^2(\xi-kt)^2}< 0.1 \nu^{1/3}
\end{align*}
in the \emph{good} region. The right-hand-side then is controlled by the decay
of $A$ and by the dissipation, which concludes the proof.

\end{proof}

Finally we turn to the control of the error in due to the convective term:
\begin{align*}
  U'' v_2= \frac{U''}{U'} U' v_2.
\end{align*}

\begin{proof}[Proof of Lemma \ref{lem:velocityerrors}]
  We remark that in the case of very small effective viscosity one cannot expect
  to control
  \begin{align}
    \label{eq:velocityerror}
    \langle A W, A U'' v_2 \rangle
  \end{align}
  in terms of the dissipation.
  Hence, we fall back to the estimates developed for the inviscid case in
  \cite{Zill3}. More precisely, let again $u=\min U'$ and define the constant
  coefficient stream function $\psi_A$ and $\psi$ as the solution of the equations
  \begin{align*}
    (\p_x^2 + u^2 (\p_z-t\p_x)^2) \psi_A =A W,
  \end{align*}
  and
  \begin{align*}
    (\p_x^2 + u^2 (\p_z-t\p_x)^2) \psi =W,
  \end{align*}
  respectively. Note that both differential operators involve constant
  coefficients and hence both $\psi$ and $\psi_A$ can be explicitly computed in
  terms of Fourier multipliers.
  
  Furthermore, integrating by parts and using that $v_2=\p_x \phi$ we control
  \eqref{eq:velocityerror} by
  \begin{align*}
    \|(\p_x, u(\p_z-t\p_x)) \psi_A\|_{L^2} \|A\| \|U''\|_{W^{1,\infty}} \|(\p_x, U' (\p_z-t\p_x)) \p_x \phi\|_{L^2},
  \end{align*}
  where $\|A\|$ denotes the $L^2$ operator norm of $A$.

  We next claim that it holds that
  \begin{align}
    \label{eq:vclaim}
    \begin{split}
    \|(\p_x, U' (\p_z-t\p_x)) \p_x \phi\|_{L^2} &\leq \|(\p_x, u(\p_z-t\p_x)) \p_x \psi\|_{L^2}\\
                                                & \leq \|A^{-1}\| \|(\p_x, u(\p_z-t\p_x)) \p_x \psi_A\|_{L^2}. 
                                              \end{split}
  \end{align}
  Assuming this claim for the moment, we may compute
  \begin{align*}
    \|(\p_x, u(\p_z-t\p_x)) \psi_A\|_{L^2}^2 = \int \frac{1}{k^2+u^2(\xi-kt)^2} |\mathcal{F}(AW)(t,k,\xi)|^2
  \end{align*}
  and hence observe that the velocity error \eqref{eq:velocityerror} can be
  estimated by
  \begin{align*}
    \int \|A\| \|A^{-1}\| \left\|\frac{U''}{u}\right\|_{W^{1,\infty}} \frac{|k|u}{k^2+u^2(\xi-kt)^2} |\mathcal{F}(AW)(t,k,\xi)|^2
  \end{align*}
  Since by assumption
  \begin{align*}
    \|A\| \|A^{-1}\| \|\frac{U''}{u}\|_{W^{1,\infty}} 
  \end{align*}
  is small this error can thus indeed be absorbed using the
  dissipation and the decay of $A(t)$.

  It remains to prove the claim \eqref{eq:vclaim} for which argue as in
  \cite{Zill3}.
  That is, we test the stream function equation
  \begin{align*}
    (\p_x^2 + (U'(\p_z-t\p-x))^2) \phi = W = (\p_x^2 + u^2 (\p_z-t\p_x)^2) \psi
  \end{align*}
  with $\phi$ (or rather $\frac{U'}{u}\phi$) to obtain that
  \begin{align*}
    \|(\p_x, U' (\p_z-t\p_x)) \phi\|_{L^2}^2 \leq \|(\p_x, U' (\p_z-t\p_x)) \phi\|_{L^2} \|(\p_x, u (\p_z-t\p_x)) \psi\|_{L^2},
  \end{align*}
  which yields the first estimate of \eqref{eq:vclaim}.
  The second estimate \eqref{eq:vclaim} immediately follows from the explicit
  characterization of $\psi$ and $\psi_A$ in terms of Fourier multipliers (which
  only differ by multiplication with the Fourier weight of $A$).
  
\end{proof}
We remark that unlike the estimates of Lemmas \ref{lem:dissipationterm} and
\ref{lem:viscosityerrors} the above estimate does not explicitly involve the
viscosity and has been obtained in the inviscid case. This lemma hence imposes
the strongest restrictions on the profile $U$ (and hence equivalently on $\mu$).
As discussed following the statement of Theorem \ref{thm:main} we do not expect
the smallness condition to be optimal, but rather a non-resonance/spectral
condition as in \cite{Zhang2015inviscid}.
However, the present stronger assumption allows for an approach in terms of
a Lyapunov functional.

\section{Localization and Non-local Interactions}
\label{sec:glue}

In this section we consider the linearized equations \eqref{eq:1} in vorticity
formulation
\begin{align*}
  \dt \omega + U(y) \p_x \omega - U'' v_2 &= \dv(\mu \nabla \omega) - \dv(\mu \nabla v_1) - \mu'' \p_x v_2,\\
  v&= \nabla^{\bot}\Delta^{-1}\omega.
\end{align*}
Unlike in Section \ref{sec:model2} we here allow for $\mu$ (and hence also $U'$) to vary by many
orders of magnitude.

Our main result of this section, Proposition \ref{prop:L2} then establishes the stability and damping results of Theorem
\ref{thm:main} in $L^2$, for which a special case had been treated in
Proposition \ref{prop:model}.

\begin{prop}
  \label{prop:L2}
  Let $\mu, U$ satisfy the assumptions of Theorem \ref{thm:main}. Then there exists a
  time-dependent family of operators $A(t)$ such that for any initial data
  $\omega_0\in L^2$ the solution $W(t)$ with that initial data satisfies
  \begin{align*}
  c \|W(t)\|_{L^2} \leq \|A(t)W(t)\|_{L^2} &\leq \|W(t)\|_{L^2}.
  \end{align*}
  Furthermore, it holds that 
  \begin{align*}
    \frac{d}{dt} \|AW \|_{L^2}^2 \leq -0.001 \left(\|(\mu (U')^2)^{1/6} W\|_{L^2}^2 + \|\sqrt{\mu}U'(\p_z-t\p_x)W\|_{L^2}^2 + \|\sqrt{U'} v\|_{L^2}^2\right).
  \end{align*}
\end{prop}

We recall that as part of the assumptions of Theorem \ref{thm:main} we require
that \eqref{eq:gradual} holds:
\begin{align*}
 \|\frac{\mu'}{\mu}\|_{W^{1,\infty}}< 0.1.
\end{align*}
This quantifies the requirement that $\mu$ may only change
gradually (but since $\R$ is unbounded it may change by many orders of magnitude
over all).
This constraint on the relative rate of change then further implies that when
restricted to any interval $I$ of suitable size, it holds that
\begin{align*}
  \frac{\max_{I}\mu}{\min_I \mu} \leq 100.
\end{align*}
Thus, if we extend the restrictions $\mu|_{I}, U|_{I}$ by constants to functions
$\mu_I, U_I$ on all of $\R$, then these extensions satisfy the assumptions of Section \ref{sec:model2}.
Thus we may ``locally'' reduce to that model setting. However, these
restrictions and extensions have to be related to the actual whole space problem \eqref{eq:1}
(see Lemma \ref{lem:normestimates}) and have to be combined to control growth of the whole
space problem (see Lemmas \ref{lem:dissipation2}, \ref{lem:velocity} and
\ref{lem:viscosity}).

Our main challenges in the following are to formalize this intuition and to
control non-local errors.
More precisely, since the velocity is non-local and so are several commutator
terms, it is not possible to just restrict $W$ and reduce estimates to the ones
of Section~\ref{sec:model2}.
Instead we will show that in the sum over all localized estimates still holds.

\subsection{Partitions and Non-local Interaction}
\label{sec:partition}
The following lemma establishes the existence of a partition of $\R$ such that
on each interval of the partition $\mu$ (and hence $U'$) is comparable to a
constant.
Furthermore, the sizes of these intervals is bounded below and hence cut-off
functions and partitions of unity corresponding to this partition have
controlled $W^{k,\infty}$ norms.
Using these partitons we may also construct extensions of the restrictions of
$\mu, U$ which satisfy the assumptions of the model setting studied in Section \ref{sec:model2}.
\begin{lem}[Partitions]
  \label{lem:partitions}
  Let $N \in \N$ and $\mu \in C^{N+2}$ be as in Theorem \ref{thm:main}. Then
  there exits a partition $(I_j)_{j\in \Z}$ of $\R$ into intervals such that
  \begin{align}
    \label{eq:localbound}
    \frac{\sup_{3 I_j}\mu}{\inf_{3 I_j} \mu} \leq 50
  \end{align}
  for all $j$, where $3 I_j$ denotes the rescaled intervals with the same
  center.
  Furthermore, the length of each interval $I_j$ is bounded below
  by $1$.

  Associated with this partition there exists a family of non-negative functions
  $\chi_j \in C_c^{\infty}$ with $\supp(\chi_j^2) \subset 3 I_j$ such that
  $\chi_j^2$ is a partition of unity.

  For each $j$ there exist $\mu_j, U_j \in C^{N+2}(\R)$ such that
  \begin{align*}
    \mu_j = \mu, U_j =U \text{ in } I_j, \\
    \mu_j U_j \equiv \text{const.} \text{ in } \R  
  \end{align*}
  and so that $\mu_j$ and $\p_y U_j$ are constant outside $3 I_j$ and
  \begin{align*}
    \frac{\max_{\R}\mu_j}{\min_{\R} \mu_j} \leq 100.
  \end{align*}
\end{lem}

\begin{proof}[Proof of Lemma \ref{lem:partitions}]
  We recall that by assumption on $\mu$ the relative rate of change
  $\frac{\mu'}{\mu}$ is bounded.
  Hence, given any two points $y_1,y_2$ we observe that
  \begin{align*}
    \frac{\mu(y_2)}{\mu(y_1)} = \exp(\ln(\mu(y_2))-\ln(\mu(y_1))) = \exp\left( \int_{y_1}^{y_2}\frac{\mu'}{\mu} dy \right)
  \end{align*}
  is bounded in terms of $|y_2-y_1|$, also when exchanging $y_{1}$ and $y_2$.

  In order to construct the intervals $I_j$ we thus pick an initial point
  $y_1=0$ and then choose $y_2>0$ (or $y_2<0$) maximally such that
  \eqref{eq:localbound} holds with $I_j=(y_1,y_2)$.
  By the above calculation the size of $I_j$ is bounded below. Therefore,
  iterating this procedure with $y_1$ chosen as a boundary point of a previously
  generated interval, we obtain the desired partition $(I_j)_{j}$ of $\R$ with the
  size of each $I_j$ bounded below.

  We remark that in this greedy procedure it is possible that for up to two choices of $j$ the
  interval $I_j$ is unbounded (in which case the above procedure only generates
  finitely many $j$). In this case we may instead impose that $y_2$ should be
  maximized under the additional constraint that $|y_2-y_1|\leq 1000$.

  It is a classical result that given such a partition of intervals there exists
  a partition of unity for which the square root of each function is still
  smooth and such that bounds on $C^k$ norms are uniform in $j$ (since the size
  of each $I_j$ is bounded below).

  Furthermore, given this partition unity we construct an extension of $\mu$ by
  \begin{align*}
    \mu_j = \mu(y_1) \sum_{l<j}\chi_l^2  + \mu \chi_j^2 + \mu(y_2) \sum_{l>j}\chi_l^2.
  \end{align*}
  The associated shear profile $U_j$ is then constructed by integrating
  \begin{align*}
    \p_y U_j := \frac{C}{\mu_j}
  \end{align*}
  with $C$ and the the constant of integration chosen such that
  $U_j(y_1)=U(y_1)$ and $\p_y U_j(y_1)= \p_y U(y_1)$.
  This then directly implies the desired bounds, where we used that the
  derivatives of the partition of unity are bounded and hence the estimate
  \eqref{eq:localbound} only possibly deteriorates by a small factor under this extension.
\end{proof}

Given these partitions we may naturally define operators acting on $\chi_jW$ by
using the results of Section \ref{sec:model2}.
\begin{defi}[Localized Fourier weights]
  \label{defi:localweights}
  Let $\chi_j^2$ be the partition of unity of Lemma \ref{lem:partitions} and let
  $\mu_j, U_j$ be the collection of viscosities and shear associated with these
  partitions.
  
  We then define $A_j$ to be the operator as given in Definition
  \ref{defi:multiplier} for $\mu, U$ replaced by $\mu_j, U_j$.
  Furthermore, we define
  \begin{align*}
    W_j: = \chi_j W
  \end{align*}
  and the energy functional
  \begin{align*}
    E(t)=\sum_j \langle A_j W_j, A_j W_j \rangle.
  \end{align*}
\end{defi}
We remark that here for each interval $I_j$ we consider the $L^2$ inner product
on $L^2(dz_j)$. However, since $\chi_j$ is compactly supported in $3 I_j$ we
observe that
\begin{align*}
  \langle A_j W_j, A_j \chi_j \p_tW \rangle_{L^2(dz_j)}= \langle A_j^2 W_j, \chi_j \p_tW  \rangle \leq \|A_j^2 W_j\|_{L^2(dz_j)} \|\chi_j \p_tW\|_{L^2(dz)},
\end{align*}
where in the last step we used that $z$ and $z_j$ agree on this support.
In view of this compatbility with $L^2(dz)$ we may hence transparently switch
between the spaces $L^2(dz_j), j \in \Z$ in several estimates and thus suppress
this formal $j$ dependence in our notation. 

Since $\chi_j^2$ is a partition of unity, the norms of $W$ and the sum of the
norms of $W_j$ are comparable.
\begin{lem}[Norm estimates]
\label{lem:normestimates}
  Let $\chi_j$ and $W_j$ be as in Definition \ref{defi:localweights}
  and suppose that $\chi_j \in C^{0}_b$.
  
  Then there exist constants $0<c_1< c_2< \infty$ such that the $L^2$ norms satisfy
  \begin{align*}
    c_1 \|W\|^2 \leq \sum_j \|W_j\|^2 \leq c_2 \|W\|^2.
  \end{align*}

  Let next $N \in \N$ and suppose that $\chi_j \in C^{N}_b$. Then there exist
  constants $d_0, \dots, d_N$ with $d_N=1$ and $c_1, c_2$ such that
  \begin{align*}
    c_1 \|W\|_{H^N}^2 \leq \sum_{l=0}^{N} \sum_j \sum_{|\alpha|=l} d_l \|\p^{\alpha}W_j\|^2 \leq c_2 \|W\|_{H^N}^2.
  \end{align*} 
\end{lem}

\begin{proof}[Proof of Lemma \ref{lem:normestimates}]
  Since $\chi_j^2$ is a partition of unity, this estimate is actually trivially
  true with $c_1=c_2=1$ and equality.

  Moreover, if $\|A_j W\| \approx \|W\|$ with constants uniform in $j$, this also
  implies that
  \begin{align*}
    \sum_j \|A_j W_j\|^2 \approx \sum_j \|W_j\|^2 = \|W\|^2.
  \end{align*}
  Here and in the following the notation $a \approx b$ states that there exist
  constants $0<c_1<c_2<\infty$ such that $c_1 a \leq b \leq c_2 a$. 

  For $N>1$ we argue by induction.
  More precisely, for any given multi-index $\alpha$ we may expand
  \begin{align*}
    \p^{\alpha}\chi_j W = \chi_j \p^{\alpha}W + \sum_{\beta+\gamma=\alpha} (\p^{\beta}\chi_j)\p^{\gamma}W.
  \end{align*}
  By the same argument as in the $L^2$ case it holds that
  \begin{align*}
    \sum_{j}\|\chi_j \p^{\alpha}W\|_{L^2}^2 = \|\p^{\alpha}W\|_{L^2}^2.
  \end{align*}
  For all other terms we note that
  \begin{align*}
    |\p^\beta \chi_j|\leq \|\chi_{j}\|_{C^{|\beta|}_b}1_{\supp(\chi_j)}
  \end{align*}
  and that the supports of the functions $\chi_j$ at most cover $\R$ twice.
  Hence, we may control
  \begin{align*}
    \sum_{\beta+\gamma=\alpha} \|(\p^{\beta}\chi_j)\p^{\gamma}W\|_{L^2}^2 \leq \sum_{m<N} \|W\|_{H^{m}}^2 \|\chi_{j}\|_{C^{N-m}_b}^2,
  \end{align*}
  which can be controlled in terms of
  \begin{align*}
    \sum_{l=0}^{N-1} \sum_j \sum_{|\alpha|=l}d_l \|\p^{\alpha}W_j\|^2 
  \end{align*}
  by the induction assumption.

  We further remark that these comparisons remain true if $W_j$ is replaced by
  $A_j W_j$.
\end{proof}

Given this definition of an energy, we next need to verify that it indeed is
a Lyapunov functional and thus study
\begin{align*}
  \frac{d}{dt}E&= \sum_j \langle \dot{A}_j W_j, A_j W_j \rangle + \sum_j \langle A_j W_j, A_j \chi_j \dt W \rangle\\
               &= \sum_j \langle \dot{A}_j W_j, A_j W_j \rangle  + \sum_{j, j'} \langle A_j W_j, A_j \chi_j (\dv_t(\mu \nabla_t \chi_{j'}W_{j'})) \rangle \\
               & \quad - \sum_j \langle A_j W_j, A_j \chi_j (\dv_t(\mu' \nabla_t v_1)- \mu'' \p_x v_2) \rangle \\
               & \quad + \sum_j \langle A_j W_j, A_j \chi_j U'' v_2 \rangle.
\end{align*}
Compared to the results of Section \ref{sec:model2} we here encounter several additional
challenges:
\begin{itemize}
\item The Biot-Savart law is \emph{non-local}. Therefore $\chi_j v$ depends on all
  $(W_{j'})_{j'}$ not just $W_j$. We thus need to compare various localizations
  of the Biot-Savart law, while at the same time also localizing in frequency.
\item The evolution of $W_j$ hence also depends on all $(W_{j'})_{j'}$. 
\item In the dissipation term we have a double sum with respect to $j$ and $j'$.
  Here we observe that for $|j-j'|\geq 2$ the support of $\chi_j$ and
  $\chi_{j'}$ are disjoint and hence we only need to consider $j' \in
  \{j-1,j,j+1\}$ (only neighbors instead of full non-local interaction as for
  the velocity).
  However, the coupling introduced by this interaction implies that we cannot
  hope to control $\langle A_j W_j, A_j W \rangle$ in terms of itself, but
  rather have to control sums over all $j$.
\end{itemize}

The following lemma generalizes Lemma \ref{lem:dissipationterm} to the present setting.
\begin{lem}[Localized dissipation estimates]
\label{lem:dissipation2}
  Let $W \in \mathcal{S}$, then it holds that
  \begin{align}
    \label{eq:diss2}
    \begin{split}
    & \quad 0.01\sum_j \langle \dot{A}_j W_j, A_j W_j \rangle + \sum_{j} \langle A_j W_j, A_j \chi_j (\dv_t(\mu \nabla_t W)) \rangle \\
   & \leq -0.01 \sum_{j} \|\sqrt{\mu} U' (\p_z-t\p_x)A_j W_j\|^2 + \|(\mu (U')^2)^{1/6} A_j W_j\|^2.
 \end{split}
  \end{align}
\end{lem}

\begin{proof}[Proof of Lemma \ref{lem:dissipation2}]
We note that in \eqref{eq:diss2} the dissipation involves $W$ and not just $W_j$
and we thus have to control the interaction with other intervals.
However, by construction only neighboring functions $\chi_j, \chi_{j'}$ with
$j'\in\{j-1,h,j+1\}$ have intersections of their support.

We thus expand
\begin{align*}
   \chi_j (\dv_t(\mu \nabla_t W)) = \dv_t(\mu \nabla_tW_j) + [\dv_t\mu \nabla_t, \chi_j] \sum_{j' \in \{j-1,j,j+1\}} W_j.
\end{align*}
Here the ``diagonal term''
\begin{align*}
  \langle A_j W_j, A_j  \dv_t(\mu \nabla_t W_j)\rangle
\end{align*}
can be controlled by using Lemma \ref{lem:dissipationterm} of Section
\ref{sec:model2}.

For the other terms we note that
\begin{align*}
  [\dv_t\mu \nabla_t, \chi_j]
\end{align*}
is a first order differential operator. In the \emph{good} region it can thus
easily be controlled by the dissipation by the same argument as in the proof of
Lemma \ref{lem:dissipationterm}.

In the \emph{bad} region we have to require that derivatives of $\chi_j$ are
not too large. As discussed in Lemma \ref{lem:partitions} this control of the
derivatives is a consequence of our assumption that $\mu$ only varies gradually
and that hence the sizes of the intervals $I_j$ is bounded below by a (large) constant.
This then implies that we can use Young's inequality to absorb
these terms into the dissipation.

This smallness is a consequence of our assumptions on $\mu$, which imply that
that each $\chi_{j}$ is supported on intervals of size at least $L$ and hence an
$n$-th order derivative is controlled in terms of $L^{-n}$, which is much
smaller than $1$.
\end{proof}

\begin{lem}[Non-local velocity estimates]
  \label{lem:velocity}
  Let $t\geq 0$, let $A, C$ and $m$ be given by Definition~\ref{defi:multiplier} and
  let $W \in H^N$ be a given function.
  Then it holds that
  \begin{align*}
    & \quad \sum_j 0.2 \langle A_jW_j, \dot{A}_j W_j \rangle - \sum_j \langle A_jW_j, A_j \chi_j (U'' v_2)  \rangle \\
    &\leq  0.1 \nu^{1/3} \sum_j\|A_jW_j\|^2 + 0.1 \sum_j\|\sqrt{\mu}\nabla_t A_jW_j\|^2
  \end{align*}

\end{lem}

\begin{proof}[Proof of Lemma \ref{lem:velocity}]
  We again observe that here the right-hand-side depends on all of $W$.
  However, unlike in Lemma \ref{lem:dissipation2} here $\chi_j v_2$ depends on
  $W_{j'}$ for all $j'$ and not just $j' \in \{j-1,j,j+1\}$.

  Instead of estimating in terms of $j'$ as in Lemma \ref{lem:dissipation2}, we
  generalize the elliptic estimates of \cite{coti2019degenerate} to the present
  setting.

  More precisely, let $\phi_j$ be the stream function generated by $W_j$:
  \begin{align*}
    \Delta_j \phi_j = W_j= \chi_j W,
  \end{align*}
  and let $\phi$ denote the stream function generated by $W$:
  \begin{align*}
    \Delta \phi= W = \sum_j \chi_j W_j.
  \end{align*}
  Then by testing the above equations with $-\phi_j$ and $-\phi$, respectively,
  we observe that
  \begin{align*}
    \|\nabla_j \phi_j\|^2 \leq \|\nabla_j \phi_j\| \|\nabla (\chi_j \phi)\|
  \end{align*}
  and
  \begin{align*}
    \|\nabla \phi\|^2 \leq \sum_j \|\nabla (\chi_j\phi)\| \|\nabla_j \phi_j\|.
  \end{align*}
  Using the fact that derivatives of $\chi_j$ are bounded, it thus follows that
  \begin{align*}
    \|\nabla \phi\|^2 \approx \sum_j \|\nabla_j \phi_j\|^2.
  \end{align*}
  Thus errors in velocity can be controlled in terms of sums of $\nabla_j
  \phi_j$ (see also Lemma \ref{lem:normestimates}). Moreover, the above argument extends to considering weighted spaces.

  In order to conclude, we note that by the definition of $U_j, \mu_j$ and $W_j$ each such contribution can be controlled in terms of the decay of the multiplier $A_j$ and the dissipation.
  Hence the velocity errors can be absorbed.
\end{proof}

 \begin{lem}[Viscosity errors]
\label{lem:viscosity}
   Let $t\geq 0$, let $A, C$ and $m$ be given by Definition~\ref{defi:multiplier} and
let $W \in H^N$ be a given function.
Then it holds that
\begin{align*}
  & \quad 0.2 \sum_j\langle A_jW_j, \dot{A}_j W_j \rangle - \sum_j \langle A_jW_j, A_j\chi_j (\dv_t(\mu' \nabla_t)v_1+ \mu'' \p_x v_2)  \rangle \\
  & \leq  0.1 \nu^{1/3} \sum_j\|A_jW_j\|^2 + 0.1 \sum_j\|\sqrt{\mu}\nabla_t A_jW_j\|^2
\end{align*}
\end{lem}

\begin{proof}[Proof of Lemma \ref{lem:viscosity}]
In order to prove these estimates we employ a combination of the methods used in
the proofs of Lemmas \ref{lem:viscosityerrors}, \ref{lem:dissipation2} and
\ref{lem:velocity}.

More precisely, we first use the structure of the Biot-Savart law to express
\begin{align*}
  (\dv_t(\mu' \nabla_t)v_1+ \mu'' \p_x v_2) 
\end{align*}
in terms of $W$ and lower order terms.
For the terms involving $W$ we can then argue analogously as in Lemma
\ref{lem:viscosityerrors}, using the decoupling of $\chi_j$ and $\chi_{j'}$ if
$j$ and $j'$ are far apart as in Lemma \ref{lem:dissipation2}.

Finally, for the remaining terms involving the velocity, we argue as in Lemma
\ref{lem:velocity} and thus reduce to estimating $\nabla_j \phi_j$ in place of
$v$. Summing over the ``diagonal'' estimates as established in Lemma
\ref{lem:viscosityerrors} then yields the result.
\end{proof}

Having establised these estimates, we are now ready to prove Proposition
\ref{prop:L2} and thus also prove part of Theorem \ref{thm:main}. An extension
of these results to higher Sobolev norms $H^N$ is given in Section
\ref{sec:proofmain}, which then completes the proof of Theorem \ref{thm:main}.

\begin{proof}[Proof of Proposition \ref{prop:L2}]
Let $\omega_0 \in L^2(dz)$ be a given initial datum, let $\mu, U$ satsify the
asssumptions of Theorem \ref{thm:main} and let $W$ denote the solution of
\eqref{Equation:W}
\begin{align*}
    &  \d_t W-U''   V_2
  =\dv_t (\mu \nabla_t  W) 
  -\dv_t(\mu' \nabla_t   v_1)
  -\mu'' \d_x v_2.
\end{align*}
with this initial data, where 
   \begin{align}
      v_1&=\frac{-U'(\d_z-t\d_x)}{\d_x^2+(U'(\d_z-t\d_x))^2}  W, \\
    \quad   v_2&=\frac{\d_x}{\d_x^2+(U'(\d_z-t\d_x))^2} W.
    \end{align}
Then by Lemma \ref{lem:partitions} there exists a parition of $\R$ into
intervals $I_j$ and an associated partition of unity $\chi_j^2$.
We then define $A_j$ and
\begin{align*}
  W_j:=\chi_j W
\end{align*}
as in Definition \ref{defi:localweights} and study the evolution of the energy
\begin{align*}
  E(t):= \sum_j \langle A_jW_j, A_j W_j \rangle.
\end{align*}
Inserting the evolution equation \eqref{Equation:W} we then have to estimate
\begin{align*}
  \frac{d}{dt}E(t) &= 2\sum_j  \langle \dot{A}_jW_j, A_j W_j \rangle \\
  &\quad +  2\sum_j  \langle A_jW_j, A_j \chi_j (\dv_t (\mu \nabla_t  W)  \rangle\\
  &\quad - 2\sum_j  \langle A_jW_j, \chi_j \dv_t(\mu' \nabla_t   v_1)\rangle\\
  &\quad - 2\sum_j  \langle A_jW_j, A_j \chi_j \mu'' \d_x v_2\rangle.
\end{align*}
Combining the estimates of each summand, derived in Lemmas
\ref{lem:normestimates} to \ref{lem:viscosity} we deduce that
\begin{align*}
   \frac{d}{dt}E(t) \leq -0.01 \sum_j \left( \nu_j^{1/3}\|W_j\|^2 +\|\sqrt{\mu_j}U_j'(\p_{z_j}-t\p_x)W_j\|^2 + u_j\|(\p_x, u_j(\p_{z_j}-t\p_x)) \psi_{A_j}\|^2 \right).
\end{align*}
Finally, we recall that $\mu_j, U'_j, z_j$ agree with $\mu, U',z$ on each interval
$I_j$ and that by Lemma \ref{lem:normestimates} the energy $E(t)$ is comparable
to $\|W(t)\|_{L^2(dz)}^2$.
This hence concludes the proof of Proposition \ref{prop:L2} where the symmetric
operator $A$ is defined such that
\begin{align*}
  \|A(t)W(t)\|^2 := E(t).
\end{align*}
\end{proof}

\section{Stability in $H^N$}
\label{sec:proofmain}

As the last step of our proof of Theorem \ref{thm:main}, in this section we
extend the stability and damping estimates in $L^2$ established in Section
\ref{sec:partition} to estimates in $H^N$.
Here we follow an inductive approach introduced in \cite{zillinger2021linear} in
the inviscid setting.
We consider the linearized equations \eqref{Equation:W}
\begin{align*}
  \d_t W &= U''   v_2 + \dv_t (\mu \nabla_t  W) 
    -\dv_t(\mu' \nabla_t   v_1)
    -\mu'' \d_x v_2 =: L W, \\
        v_1&=\frac{-U'(\d_z-t\d_x)}{\d_x^2+(U'(\d_z-t\d_x))^2}  W, \\
      v_2&=\frac{\d_x}{\d_x^2+(U'(\d_z-t\d_x))^2} W,
\end{align*}
where we introduced the time-dependent linear operator $L$ for brevity of notation. 
We remark that derivatives with respect to $x$ can be identified with
multiplication by $ik$, since the linearized equations decouple with respect to
$k$. Hence higher derivatives in $x$ can be estimated using the $L^2$ energy. In
the following we hence only consider derivatives with respect to $z$.
Applying $N$ derivatives to \eqref{Equation:W} we obtain that
\begin{align}
  \label{eq:WN}
  \d_t \p_z^N W = L \p_z^N W + [L,\p_z^N]W. 
\end{align}
In the following lemma we then that the commutator term can be considered
an error term involving fewer than $N$ derivatives, while $L \p_z^N W$ can be
treated in the same way as in the $L^2$ estimate. In this sense the $L^2$
estimate forms the core of our argument.

\begin{prop}
  \label{prop:HN}
  Let $\mu,U$ satisfy the assumptions of Theorem \ref{thm:main}.
  In particular, let $N \in \N$ and suppose that $\p_z \ln(\mu) \in W^{N+1,\infty}$.
  Let $A$ be as in Proposition \ref{prop:L2}, then there exist constants
  $c_0,c_1,\dots, c_N>0$ depending only on the $W^{k,\infty}$ norms of
  $\p_z \ln(\mu)$ such that
  \begin{align*}
    E_N(t)= \sum_{l\leq N} c_l\langle A \p_z^lW, A \p_z^l W \rangle
  \end{align*}
  is a Lyapunov functional and satisfies
  \begin{align*}
    \frac{d}{dt} E(t)\leq -0.01 (\|\sqrt{\mu}(\p_z-t\p_x) \p_z^N W\|_{L^2}^2 + \|(\mu (U')^2)^{\frac{1}{6}}\p_z^NW\|_{L^2}^2).
  \end{align*}
\end{prop}

We remark that here we only require that the $W^{N+1,\infty}$ norm is finite.
Only the $W^{1,\infty}$ needs to be small in order to establish the $L^2$
stability estimate.

\begin{proof}[Proof of Prosposition \ref{prop:HN}]
The case $N=0$ has been established in Proposition \ref{prop:L2} with $c_0=1$.
We hence aim to proceed by induction. Hence, suppose that the estimates have
been established for the case $N-1$ and consider
\begin{align*}
  E_N(t)= c_N \langle A \p_z^N W, A \p_z^N W \rangle + E_{N-1}(t)+ E_{N-2}(t)+\dots +E_{0}(t)
\end{align*}
with $c_N$ to be determined later.

Then by the induction assumption it holds that
\begin{align}
  \label{eq:AN1}
  \frac{d}{dt} E_{l}(t) \leq -0.01 (\|\sqrt{\mu}(\p_z-t\p_x) \p_z^{l} W\|_{L^2}^2 + \|(\mu (U')^2)^{\frac{1}{6}}\p_z^{l}W\|_{L^2}^2)
\end{align}
for all $0\leq l\leq N-1$
In particular, all derivatives of $W$ up to order $N-1$ can be controlled by the induction assumption.
We thus turn to the control of the ``leading order'' term involving $\p_z^NW$.
Here, by the $L^2$ estimates of Proposition \ref{prop:L2} it holds that
\begin{align}
  \label{eq:AN}
  \begin{split}
  \frac{d}{dt} c_N \langle A \p_z^N W, A \p_z^N W \rangle &= 2c_N \langle \dot{A} \p_z^N W, A \p_z^N W \rangle + 2c_n \langle A \p_z^N W, A L \p_z^N W \rangle \\
                                                          & \quad +  2c_N \langle A \p_z^N W, A [L, \p_z^N] W \rangle \\
                                                          &\leq -0.01 c_N (\|\sqrt{\mu}(\p_z-t\p_x) \p_z^{N-1} W\|_{L^2}^2 + \|(\mu (U')^2)^{\frac{1}{6}}\p_z^{N-1}W\|_{L^2}^2)\\
                                                          & \quad +  2c_N \langle A \p_z^N W, A [L, \p_z^N] W \rangle.
                                                        \end{split}
\end{align}
Combining the estimates \eqref{eq:AN} and \eqref{eq:AN1} it thus suffices to
show that for a suitable choice of $c_N$ we may absorb the commutation error
\begin{align*}
  2c_N \langle A \p_z^N W, A [L, \p_z^N] W \rangle.
\end{align*}
into the decay
\begin{align*}
  -0.01 c_N (\|\sqrt{\mu}(\p_z-t\p_x) \p_z^{N-1} W\|_{L^2}^2 + \|(\mu (U')^2)^{\frac{1}{6}}\p_z^{N-1}W\|_{L^2}^2) \\
  -0.01 \sum_{l<N} (\|\sqrt{\mu}(\p_z-t\p_x) \p_z^{l} W\|_{L^2}^2 + \|(\mu (U')^2)^{\frac{1}{6}}\p_z^{l}W\|_{L^2}^2)
\end{align*}
Let us first discuss the main dissipation term of $L$. Here we may iteratively expand
\begin{align*}
  [\dv_t(\mu \nabla_t), \p_z^N]W &= [\dv_t(\mu \nabla_t), \p_z] \p_z^{N-1}W \\
  &\quad + [\dv_t(\mu \nabla_t), \p_z^{N-1}] \p_z^{1}W + [[\dv_t(\mu \nabla_t), \p_z^{N-1}], \p_z]W \\
                                &= \sum_{l<N} E_l \p_z^lW,
\end{align*}
where the operators $E_l$ are second order elliptic operators whose coefficient functions may
be explicitly computed in terms of derivatives of $U'$ and $\mu$ up to order
$N-l$.
In order to estimate 
\begin{align*}
  \langle A\p_z^N W, A [\dv_t(\mu \nabla_t), \p_z^N]W \rangle = \sum_{l<N} \langle A \p_z^N, A E_l \p_z^l W \rangle
\end{align*}
we may thus argue as in the proof of Lemma \ref{lem:dissipation2} and control
\begin{align}
  \label{eq:dissipationN}
  c_N \sum_{l<N} \langle A \p_z^N, A E_l \p_z^l W \rangle &\leq c_N \|\sqrt{\mu}\nabla_t A \p_z^NW\|_{L^2} \sum_{l} d_l \|\sqrt{\mu}\nabla_t \p_z^kl W\|_{L^2}.
\end{align}

Similarly we may iterative expand the equation satisfied by derivatives of the
stream function
\begin{align*}
  \Delta_t \p_z^N \phi = \p_z^N W + [\Delta_t, \p_z^N]\phi
\end{align*}
and thus obtain that
\begin{align*}
  \p_z^N \phi = \Delta_t^{-1} \p_z^NW + \Delta_t^{-1} \sum_{l<N} \tilde{E_l} \Delta_t^{-1} \p_z^l W, 
\end{align*}
where the second order operators $\tilde{E}_l$ may again be explicitly computed.
Thus, we may argue as in the proofs of Lemmas \ref{lem:velocity} and
\ref{lem:viscosity} and again use Hölder's and Young's inequality to control
\begin{align}
  \label{eq:velocityN}
  \begin{split}
&\quad  \langle A\p_z^N W, A [\dv_t(\mu' \nabla_t)U'(\p_z-t\p_x)\Delta_t^{-1}- \mu'' \p_x \Delta_t^{-1} + U'' \p_x \Delta_t^{-1}, \p_z^N]W \rangle \\
&\leq  (\|\sqrt{\mu}(\p_z-t\p_x) \p_z^{N} W\|_{L^2} + \|(\mu (U')^2)^{\frac{1}{6}}\p_z^{N}W\|_{L^2}) \\
&\quad \quad \sum_{l} (\|\sqrt{\mu}(\p_z-t\p_x) \p_z^{l} W\|_{L^2} + \|(\mu (U')^2)^{\frac{1}{6}}\p_z^{l}W\|_{L^2}).
\end{split}
\end{align}

We may thus conclude our estimate by using Young's inequality. More precisely,
we first apply Young's inequality to the estimates \eqref{eq:dissipationN}
and \eqref{eq:velocityN} so that the contributions due to $\p_z^NW$ can be
bounded by
\begin{align*}
  0.00001 c_{N} (\|\sqrt{\mu}(\p_z-t\p_x) \p_z^{N} W\|_{L^2}^2 + \|(\mu (U')^2)^{\frac{1}{6}}\p_z^{N}W\|_{L^2}^2) 
\end{align*}
and can thus be absorbed into the decay in estimate \eqref{eq:AN}.
Then, choosing $c_N$ sufficiently small the remaining terms obtained in the
application of Young's inequality can be absorbed into the decay by \eqref{eq:AN1}.
This concludes the proof.
\end{proof}

\section*{Acknowledgments}
Funded by the Deutsche Forschungsgemeinschaft (DFG, German Research Foundation) – Project-ID 258734477 – SFB 1173.

\bibliography{references.bib,citations2.bib}

\begin{thebibliography}{MSHZ20}

\bibitem[BG81]{BG}
Steven~J Barker and Douglas Gile.
\newblock Experiments on heat-stabilized laminar boundary layers in water.
\newblock {\em Journal of Fluid Mechanics}, 104:139--158, 1981.

\bibitem[BGM17]{bedrossian2017stability}
Jacob Bedrossian, Pierre Germain, and Nader Masmoudi.
\newblock On the stability threshold for the 3d {C}ouette flow in {S}obolev
  regularity.
\newblock {\em Annals of Mathematics}, pages 541--608, 2017.

\bibitem[BM15]{bedrossian2013inviscid}
Jacob Bedrossian and Nader Masmoudi.
\newblock Inviscid damping and the asymptotic stability of planar shear flows
  in the 2d {E}uler equations.
\newblock {\em Publications math{\'e}matiques de l'IH{\'E}S}, 122(1):195--300,
  2015.

\bibitem[BMV16]{bedrossian2016enhanced}
Jacob Bedrossian, Nader Masmoudi, and Vlad Vicol.
\newblock Enhanced dissipation and inviscid damping in the inviscid limit of
  the {N}avier--{S}tokes equations near the two dimensional {C}ouette flow.
\newblock {\em Archive for Rational Mechanics and Analysis}, 219(3):1087--1159,
  2016.

\bibitem[BVW18]{bedrossian2016sobolev}
Jacob Bedrossian, Vlad Vicol, and Fei Wang.
\newblock The {S}obolev stability threshold for {2D} shear flows near
  {C}ouette.
\newblock {\em Journal of Nonlinear Science}, 28(6):2051--2075, 2018.

\bibitem[Cra69]{Craik}
Alex~DD Craik.
\newblock The stability of plane couette flow with viscosity stratification.
\newblock {\em Journal of Fluid Mechanics}, 36(4):685--693, 1969.

\bibitem[CZZ19]{coti2019degenerate}
Michele Coti~Zelati and Christian Zillinger.
\newblock On degenerate circular and shear flows: the point vortex and power
  law circular flows.
\newblock {\em Communications in Partial Differential Equations},
  44(2):110--155, 2019.

\bibitem[Dra02]{Drazin}
Philip~G Drazin.
\newblock {\em Introduction to hydrodynamic stability}, volume~32.
\newblock Cambridge university press, 2002.

\bibitem[DWZ20]{deng2020stability}
Wen Deng, Jiahong Wu, and Ping Zhang.
\newblock Stability of couette flow for 2d {B}oussinesq system with vertical
  dissipation.
\newblock {\em arXiv preprint arXiv:2004.09292}, 2020.

\bibitem[EW15]{elgindi2015sharp}
Tarek~M Elgindi and Klaus Widmayer.
\newblock Sharp decay estimates for an anisotropic linear semigroup and
  applications to the surface quasi-geostrophic and inviscid {B}oussinesq
  systems.
\newblock {\em SIAM Journal on Mathematical Analysis}, 47(6):4672--4684, 2015.

\bibitem[GS14]{GS}
Rama Govindarajan and Kirti~Chandra Sahu.
\newblock Instabilities in viscosity-stratified flow.
\newblock {\em Annual review of fluid mechanics}, 46:331--353, 2014.

\bibitem[HB87]{HB}
AP~Hooper and WGC Boyd.
\newblock Shear-flow instability due to a wall and a viscosity discontinuity at
  the interface.
\newblock {\em Journal of Fluid Mechanics}, 179:201--225, 1987.

\bibitem[Hei85]{Heisenberg}
Werner Heisenberg.
\newblock {\"U}ber stabilit{\"a}t und turbulenz von
  fl{\"u}ssigkeitsstr{\"o}men.
\newblock In {\em Original Scientific Papers Wissenschaftliche
  Originalarbeiten}, pages 31--81. Springer, 1985.

\bibitem[HL20]{HL}
Zihui He and Xian Liao.
\newblock Solvability of the two-dimensional stationary incompressible
  inhomogeneous navier-stokes equations with variable viscosity coefficient.
\newblock {\em arXiv preprint arXiv:2005.13277}, 2020.

\bibitem[IJ18]{ionescu2018inviscid}
Alexandru Ionescu and Hao Jia.
\newblock Inviscid damping near shear flows in a channel.
\newblock {\em arXiv preprint arXiv:1808.04026}, 2018.

\bibitem[Jia19]{jia2019linear}
Hao Jia.
\newblock Linear inviscid damping in {G}evrey spaces.
\newblock {\em arXiv preprint arXiv:1904.01188}, 2019.

\bibitem[JRR84]{JRR}
Daniel~D Joseph, Michael Renardy, and Yuriko Renardy.
\newblock Instability of the flow of two immiscible liquids with different
  viscosities in a pipe.
\newblock {\em Journal of Fluid Mechanics}, 141:309--317, 1984.

\bibitem[Lin44]{Lin}
CC~Lin.
\newblock On the stability of two-dimensional parallel flows.
\newblock {\em Proceedings of the National Academy of Sciences of the United
  States of America}, 30(10):316, 1944.

\bibitem[Lis20]{liss2020sobolev}
Kyle Liss.
\newblock On the {S}obolev stability threshold of 3d {C}ouette flow in a
  uniform magnetic field.
\newblock {\em Communications in Mathematical Physics}, pages 1--50, 2020.

\bibitem[LX19]{lin2019metastability}
Zhiwu Lin and Ming Xu.
\newblock Metastability of kolmogorov flows and inviscid damping of shear
  flows.
\newblock {\em Archive for Rational Mechanics and Analysis}, 231(3):1811--1852,
  2019.

\bibitem[LZ11]{Lin-Zeng}
Zhiwu Lin and Chongchun Zeng.
\newblock {Inviscid dynamical structures near {C}ouette Flow}.
\newblock {\em Archive for rational mechanics and analysis}, 200(3):1075--1097,
  2011.

\bibitem[MSHZ20]{masmoudi2020stability}
Nader Masmoudi, Belkacem Said-Houari, and Weiren Zhao.
\newblock Stability of {C}ouette flow for 2d {B}oussinesq system without
  thermal diffusivity.
\newblock {\em arXiv preprint arXiv:2010.01612}, 2020.

\bibitem[PV91]{PV}
Henry Power and Miguel Villegas.
\newblock Weakly nonlinear instability of the flow of two immiscible liquids
  with different viscosities in a pipe.
\newblock {\em Fluid dynamics research}, 7(5-6):215, 1991.

\bibitem[TW19]{tao20192d}
Lizheng Tao and Jiahong Wu.
\newblock The 2d {B}oussinesq equations with vertical dissipation and linear
  stability of shear flows.
\newblock {\em Journal of Differential Equations}, 267(3):1731--1747, 2019.

\bibitem[Wid18]{widmayer2018convergence}
Klaus Widmayer.
\newblock Convergence to stratified flow for an inviscid 3d {B}oussinesq
  system.
\newblock {\em Communications in Mathematical Sciences}, 16(6):1713--1728,
  2018.

\bibitem[WZZ17]{WZZkolmogorov}
D.~{Wei}, Z.~{Zhang}, and W.~{Zhao}.
\newblock {Linear inviscid damping and enhanced dissipation for the Kolmogorov
  flow}.
\newblock {\em ArXiv e-prints}, November 2017.

\bibitem[WZZ18]{Zhang2015inviscid}
Dongyi Wei, Zhifei Zhang, and Weiren Zhao.
\newblock Linear inviscid damping for a class of monotone shear flow in
  {S}obolev spaces.
\newblock {\em Communications on Pure and Applied Mathematics}, 71(4):617--687,
  2018.

\bibitem[Yih67]{Yih}
Chia-Shun Yih.
\newblock Instability due to viscosity stratification.
\newblock {\em Journal of Fluid Mechanics}, 27(2):337--352, 1967.

\bibitem[YL18]{yang2018linear}
Jincheng Yang and Zhiwu Lin.
\newblock Linear inviscid damping for {C}ouette flow in stratified fluid.
\newblock {\em Journal of Mathematical Fluid Mechanics}, 20(2):445--472, 2018.

\bibitem[Zil17]{Zill3}
Christian Zillinger.
\newblock Linear inviscid damping for monotone shear flows.
\newblock {\em Trans. Amer. Math. Soc.}, 369(12):8799--8855, 2017.

\bibitem[Zil21]{zillinger2021linear}
Christian Zillinger.
\newblock Linear inviscid damping in sobolev and gevrey spaces.
\newblock {\em Nonlinear Analysis}, 213:112492, 2021.

\end{thebibliography}
\bibliographystyle{alpha}

\end{document}